\documentclass[reqno,12pt]{amsart}  
\theoremstyle{plain}
\newtheorem{theorem}{Theorem}[section]
\newtheorem{lemma}[theorem]{Lemma}
\newtheorem{remark}[theorem]{Remark}
\newtheorem{proposition}[theorem]{Proposition}

\numberwithin{equation}{section}
\allowdisplaybreaks[1]

\theoremstyle{definition}

\newtheorem{definition}[theorem]{Definition}

\theoremstyle{remark}

\newcommand{\be}{{\mathbf e}}

\newcommand{\bU}{{\mathbf U}}

\newcommand{\bT}{{\mathbf T}}

\newcommand{\cD}{{\mathcal D}}

\newcommand{\cH}{{\mathcal H}}
\newcommand{\cL}{{\mathcal L}}
\newcommand{\cM}{{\mathcal M}}
\newcommand{\cR}{{\mathcal R}}
\newcommand{\cK}{{\mathcal K}}

\newcommand{\cS}{{\mathcal S}}
\newcommand{\cE}{{\mathcal E}}

\newcommand{\cU}{{\mathcal U}}
\newcommand{\cX}{{\mathcal X}}
\newcommand{\cY}{{\mathcal Y}}
\newcommand{\cO}{{\mathcal O}}
\newcommand{\cG}{{\mathcal G}}
\newcommand{\D}{{\mathbb D}}

\newcommand{\C}{{\mathbb C}}

\newcommand{\sbm}[1]{\left[\begin{smallmatrix} #1
		\end{smallmatrix}\right]}

\begin{document}

\title[de Branges-Rovnyak spaces]{de Branges-Rovnyak spaces and norm-constraint
interpolation}
\author[J. A. Ball]{Joseph A. Ball}
\address{Department of Mathematics,
Virginia Tech,
Blacksburg, VA 24061-0123, USA}
\email{joball@math.vt.edu}
\author[V. Bolotnikov]{Vladimir Bolotnikov}
\address{Department of Mathematics,
The College of William and Mary,
Williamsburg VA 23187-8795, USA}
\email{vladi@math.wm.edu}

\begin{abstract}
For $S$ a contractive analytic operator-valued function on the 
unit disk ${\mathbb D}$, de Branges and Rovnyak associate a  
Hilbert space of analytic functions $\cH(S)$. A companion survey 
provides equivalent definitions and basic properties of these spaces 
as well as applications to function theory and operator theory.
The present survey brings to the fore more recent applications to a 
variety of more elaborate function theory problems, including
$H^\infty$-norm constrained interpolation, connections with the 
Potapov method of Fundamental Matrix Inequalities, parametrization 
for the set of all solutions of an interpolation problem, variants of 
the Abstract Interpolation Problem of Katsnelson, Kheifets, and 
Yuditskii, boundary behavior and boundary interpolation in de 
Branges-Rovnyak spaces themselves, and 
extensions to multivariable and Kre\u{\i}n-space settings.
\end{abstract}


\maketitle
\tableofcontents

\section{Introduction}
 In the late 1960s and early 1970s, 
Louis de Branges and James Rovnyak introduced and studied spaces of 
vector-valued holomorphic functions on the open unit disk $\D$ associated with 
what is now called a Schur-class function $S \in \cS(\cU, \cY)$ 
(i.e., a holomorphic function $S$ on the unit disk with values equal 
to contraction operators between Hilbert coefficient spaces $\cU$ and 
$\cY$. Motivation for the study of these spaces came from   
quantum scattering theory (see \cite{dBR1, dB77, dBS}), and operator model 
theory for Hilbert space contraction operators and the invariant 
subspace problem (see \cite[Appendix]{dBR1} and \cite{dBR2}).

\smallskip

Interpolation by Schur-class functions is an older area which appeared first
within geometric function theory. Over the years there have been a variety of approaches
to the study of Schur-class functions and associated interpolation problems (e.g., Schur 
algorithm,
iterated one-step extension procedures, transfer-function realization techniques, 
the Grassmannian Kre\u{\i}n-space geometry approach, reproducing kernel Hilbert space 
methods,  
and commutant-lifting methods to mention a few).  The general topic for this survey 
article is de Branges-Rovnyak spaces;
hence the focus here is only on those approaches which rely to some extent on de 
Branges-Rovnyak spaces.

There are now at least three 
distinct ways of introducing the de Branges-Rovnyak spaces: 
\begin{enumerate}
\item  the 
original definition of de Branges and Rovnyak (as the complementary 
space of $S \cdot H^{2}$), 
\item  as the range of the Toeplitz defect 
operator with lifted norm, or 
\item  as the reproducing kernel Hilbert 
space with reproducing kernel given by the de Branges-Rovnyak 
positive kernel.  
\end{enumerate}

\section{de Branges-Rovnyak spaces}

In what follows, the symbol $\cL(\cU,\cY)$ stands for the space of bounded linear 
operators mapping a Hilbert space $\cU$ into a Hilbert space $\cY$, abbreviated to 
$\cL(\cY)$ in case $\cU=\cY$.   The notation $H^2(\cY)$  is used to 
denote the standard Hardy space of $\cY$-valued functions on the open unit disk $\D$ 
with square-summable sequence of Taylor coefficients while 
$\cS(\cU,\cY)$ denotes the Schur class of functions analytic on $\D$
with values equal to contractive operators in $\cL(\cU,\cY)$. The de Branges-Rovnyak
space $\cH(S)$ associated with a given Schur-class function $S\in\cS(\cU,\cY)$
was originally defined as the complementary  space of $S \cdot H^{2}$ by the prescription
\begin{equation}
\cH(S) =\{ f \in H^{2}(\cU) \colon \| f\|^{2}_{\cH(S)}: = \sup_{g \in
H^{2}(\cU)}\{
\|f + Sg \|^{2}_{H^{2}(\cY)}  - \| g \|^{2}_{H^{2}(\cU)}\} < \infty\}.\label{def1HS}
\end{equation}
In particular, it follows from \eqref{def1HS} that $\|f\|_{\cH(K_{S})}\ge
\|f\|_{H^2(\cY)}$ for every $f\in\cH(K_{S})$, i.e., that $\cH(K_{S})$ is contained in
$H^2(\cY)$ contractively.

\smallskip

Two equivalent definitions of de Branges-Rovnyak spaces (more convenient in certain 
contexts)
involve the notion of a {\em reproducing kernel Hilbert space} which will be now recalled.
\subsection{Reproducing kernel Hilbert spaces}  \label{S:RKHS}
 A {\em reproducing kernel Hilbert space} (RKHS) is a Hilbert space whose elements 
are functions on some set $\Omega$
 with values in a coefficient Hilbert space, say $\cY$, such that the
 evaluation map $\be(\omega) \colon f \mapsto f(\omega)$ is
 continuous from $\cH$ into $\cY$ for each $\omega \in \Omega$.
 Associated with any such space is a positive $\cL(\cY)$-valued
 kernel on $\Omega$, i.e., a function $K \colon \Omega \times \Omega
 \to \cL(\cY)$ with the positive-kernel property
 \begin{equation}  \label{posker}
 \sum_{i,j=1}^{N} \langle K(\omega_{i}, \omega_{j}) y_{j}, y_{i}
 \rangle_{\cY} \ge 0
 \end{equation}
for any choice of finitely many points $\omega_{1}, \dots, \omega_{N}
 \in \Omega$ and vectors $y_{1}, \dots, y_{N} \in \cY$,
which ``reproduces'' the values of the functions in $\cH$ in the sense
 that
 \begin{enumerate}
     \item[(i)] the function $\omega \mapsto K(\omega, \zeta) y$
     is in $\cH$ for each $\zeta \in \Omega$ and $y \in \cY$, and
     \item[(ii)] the reproducing formula
     $$
       \langle f, K(\cdot, \zeta) y \rangle_{\cH} = \langle
       f(\zeta), y \rangle_{\cY}
       $$
  holds for all $f \in \cH$, $\zeta \in \Omega$, and $y \in \cY$.
  \end{enumerate}
 An early thorough treatment of RKHSs (for
 the case $\cY = {\mathbb C}$) is the paper of Aronszajn
 \cite{aron}; a good recent treatment is in the book    
 \cite{AgMcC}, while the recent paper \cite{BV-formal}
 formulates more general settings (formal commuting or noncommuting
 variables).
     
\smallskip
     
 Given a pair of reproducing kernel Hilbert spaces $\cH(K_{1})$ and
 $\cH(K_2)$ where say $\cH(K_{1})$ consists of functions with values in
 $\cU$ and $\cH(K_2)$ consists of functions with values in $\cY$,  
 an object of much interest for operator theorists is the space of   
 multipliers $\cM(K_1, K_2)$ consisting of $\cL(\cU, \cY)$-valued
 functions $F$ on $\Omega$ with the property that the multiplication 
 operator
 $$
    M_{F} \colon f(\zeta) \mapsto F(\zeta) f(\zeta)
 $$
 maps $\cH(K_1)$ into $\cH(K_2)$.  The simple computation
 \begin{align*}
 \langle M_{F} f, K(\cdot, \zeta) y \rangle_{\cH(K_2)} &= \langle F(\zeta)
 f(\zeta), y \rangle_{\cY}\notag\\ &= \langle f(\zeta), F(\zeta)^{*} y  
 \rangle_{\cU}= \langle f, K_{1}(\cdot, \zeta) F(\zeta)^{*} y
 \rangle_{\cH(K_{1})}
 \end{align*}
 shows that  
 \begin{equation}   \label{ShieldsWallen}
    (M_{F})^{*} \colon K_2(\cdot, \zeta) y \mapsto K_{1}(\cdot, \zeta)
    F(\zeta)^{*} y.
 \end{equation}
Therefore
$$
\langle (I-M_FM_F^*)K_2(\cdot,\zeta)y, \, K_2(\cdot,\omega)y^\prime\rangle_{\cH(K_2)}=
\langle(K_2(\omega,\zeta)-F(z)K_1(\omega,\zeta)F(\zeta)^*)y, \, y^\prime\rangle_{\cY}
$$
which implies  that $F$ is a contractive multiplier from $\cH(K_1)$ to $\cH(K_2)$ if and 
only if
the kernel  $K_2(\omega,\zeta)-F(z)K_1(\omega,\zeta)F(\zeta)^*$ is positive on 
$\Omega\times\Omega$.  
Letting $K_1(\omega,\zeta)\equiv I_\cY$ and performing a rescaling leads to the following 
proposition 
\cite{beabur}.

\begin{proposition}
A function $F: \, \Omega\to \cY$ belongs to $\cH(K)$ with $\|F\|_{\cH(K)}\le \gamma$ if 
and only if
the kernel $K(\omega,\zeta)-\gamma^{-2}F(z)F(\zeta)^*$ is positive on 
$\Omega\times\Omega$.
\label{bourb}
\end{proposition}

\subsection{The Toeplitz operator characterization of $\cH(K)$}
A first example of a reproducing kernel Hilbert space is the
 Hardy space $H^{2}(\cY)$ of $\cY$-valued functions on the open unit disk $\D$
with square-summable sequence of Taylor coefficients. This space can be viewed as 
a RKHS with the Szeg\H{o} kernel tensored
 with the identity operator on $\cY$:  $k_{\rm Sz}(z,\zeta) I_{\cY}$ where
 $k_{\rm Sz}(z,\zeta) = \frac{1}{1 - z \overline{\zeta}}$.  The space of
 multipliers $\cM(k_{\rm Sz} I_{\cU}, k_{\rm Sz} I_{\cY})$ between
 two Hardy spaces can be identified with the space
 $H^{\infty}(\cU, \cY)$ of bounded analytic functions on
 ${\mathbb D}$ with values in $\cL(\cU, \cY)$ while the set of contractive multipliers is 
identified with the Schur class $\cS(\cU,\cY)$. Indeed, for $S \in \cS(\cU, \cY)$ and 
for any $f\in H^2(\cU)$, 
\begin{align}
\|Sf\|^2_{H^2(\cY)}&=\sup_{0<r<1}\frac{1}{2\pi}\int_{0}^{2\pi}\|S(re^{it})f(re^{it})\|^2_{\cY}dt
\notag\\
&\le 
\sup_{0<r<1}\frac{1}{2\pi}\int_{0}^{2\pi}\|S(e^{it})f(e^{it})\|^2_{\cY}dt=\|f\|^2_{H^2(\cU)}
\label{ineq}
\end{align}
which shows that $M_S$ is a contraction from $H^2(\cU)$ to $H^2(\cY)$. The general
complementation theory applied to the contractive operator $M_S$
provides the {\em characterization of $\cH(K_{S})$ as the operator range
\begin{equation}
{\cH}(K_{S})={\rm Ran}(I-M_SM^*_S)^{\frac{1}{2}}
\label{def2HS}
\end{equation}
with the lifted norm
\begin{equation}
\|(I-M_SM^*_S)^{\frac{1}{2}}f\|_{\cH(K_{S})}=\|(I-\pi)f \|_{H^2(\cY)}
\label{1.14u}
\end{equation}
for all $f\in H^2(\cY))$} where $\pi$ is the orthogonal
projection onto ${\rm Ker}(I-M_SM^*_S)^{\frac{1}{2}}$.
Upon setting $f=(I-M_SM^*_S)^{\frac{1}{2}}h$ in \eqref{1.14u} one gets 
\begin{equation}
\| (I - M_{S}M_{S}^{*})h\|_{\cH(K_{S})}=\langle (I - M_{S} M_{S}^{*})h, \,
h\rangle_{H^2(\cY)}.
\label{1.15u}
\end{equation}

\subsection{Reproducing kernel characterization of $\cH(S)$}

As a result of the general identity \eqref{ShieldsWallen}, 
 \begin{equation}  \label{ts*}
 M_{S}^{*} \colon k_{\rm Sz}(\cdot, \zeta) y \mapsto k_{\rm
        Sz}(\cdot, \zeta) S(\zeta)^{*} y
 \end{equation} 
  and hence 
\begin{equation}  \label{ts*1}
(I - M_{S} M_{S}^{*}) k_{\rm Sz}(z ,\zeta) y=K_S(z,\zeta)y
\end{equation}
where  
\begin{equation}
K_S(z,\zeta)=(I-S(z)S(\zeta)^*)k_{\rm Sz}(z ,\zeta)=
\frac{I_\cY-S(z)S(\zeta)^*}{1-z\bar{\zeta}}
\label{deBRkernel}
\end{equation}
is the {\em de Branges-Rovnyak kernel} associated to the given $S\in\cS(\cU,\cY)$.
Application of inequality \eqref{ineq} to $f = \sum_{j=1}^{N} k_{\rm Sz}(\cdot, w_{j})y_{j} \in
H^{2}(\cY)$ leads one, on account of \eqref{ts*1}, to 
$$
\sum_{i,j=1}^N \langle K_S(\omega_i,\omega_j)y_j, \, y_i\rangle_{\cY} \ge 0,
$$
and it follows that $K_{S}$ is a positive kernel on $\D\times \D$. 
Combining the characterization \eqref{def2HS} and equality \eqref{ts*1} one can see that 
$K_{S}(\cdot, \zeta) y \in \cH(S)$ for each $\zeta \in
  {\mathbb D}$ and $y \in \cY$, and also, for $f = (I - M_{S} M_{S}^{*}) f_{1} \in
  \cH(S)$,
  \begin{align*}
  \langle f, K_{S}(\cdot, \zeta) y \rangle_{\cH(S)}&  = \langle f, (I -
  M_{S} M_{S}^{*}) (k_{\rm Sz}(\cdot, \zeta) y \rangle_{\cH(S)} \\
 & = \langle f, k_{\rm Sz}(\cdot, \zeta) y \rangle_{H^{2}(\cY)}= \langle f(\zeta), y 
\rangle_{\cY}
  \end{align*}
  from which it follows that {\em $\cH(S)$ is a reproducing kernel 
  Hilbert space with reproducing kernel equal to $K_{S}(z,\zeta)$ 
  \eqref{deBRkernel}.}
 This characterization of $\cH(S)$ turns out to be quite convenient in 
interpolation and realization contexts. This section concludes by recording several useful 
facts 
concerning de Branges-Rovnyak spaces collected in the following theorem.
\begin{theorem}  \label{T:H(S)}
    If $S \in \cS(\cU, \cY)$, the space $\cH(S)$ has the following  
    properties:
\begin{enumerate}
    \item  $\cH(S)$ is a linear
space, indeed a reproducing kernel Hilbert space with reproducing  
kernel $K_{S}(z,w)$ given by   
$$
  K_{S}(z,w) = \frac{I - S(z) S(w)^{*}}{1 - z \overline{w}}.
$$
\item The space $\cH(S)$ is invariant under the backward-shift operator
\begin{equation}  \label{RS}
  R_{0} \colon f(z) \mapsto [f(z) - f(0)]/z
\end{equation}
and the following norm estimate holds:
\begin{equation}   \label{norm-est}
    \| R_{0} f\|^{2}_{\cH(S)} \le \| f \|^{2}_{\cH(S)} - \| f(0)
    \|^{2}_{\cY}.
\end{equation}
Moreover, equality holds in \eqref{norm-est} for all $f \in \cH(S)$
if and only if $\cH(S)$ has the property
$$
 S(z) \cdot u \in \cH(S) \Rightarrow S(z) \cdot u \equiv 0.
$$
\item
For any $u \in \cU$, the function $R_0(Su)$
is in $\cH(S)$.  If one lets $\tau \colon \cU \to \cH(S)$
denote the operator
\begin{equation}  \label{tau}
  \tau \colon u \mapsto R_0(Su)=\frac{S(z) - S(0)}{z} \,  u,
\end{equation}
then the adjoint $R_{0}^{*}$ of the operator $R_{0}$ \eqref{RS} on $\cH(S)$ is given by
\begin{equation}   \label{RS*}
  R_{0}^{*} \colon f(z) \mapsto z f(z) - S(z) \cdot \tau^{*}(f)
\end{equation}
with the following formula for the norm holding:
\begin{equation}   \label{RS*norm}
    \| R_{0}^{*} f\|^{2}_{\cH(S)} = \| f \|^{2}_{\cH(S)} - \| 
    \tau^{*}(f) \|^{2}_{\cU}.
\end{equation}

\item
Let  $\bU_{S}$ be the colligation matrix given by
\begin{equation}   \label{bUS}
 \bU_{S} = \begin{bmatrix}  A_{S} & B_{S} \\ C_{S} & D_{S}
\end{bmatrix}: = \begin{bmatrix}
     R_{0} & \tau \\ \be(0) & S(0) \end{bmatrix}  \colon
     \begin{bmatrix} \cH(S) \\ \cU \end{bmatrix} \to \begin{bmatrix}
         \cH(S) \\ \cY \end{bmatrix}
\end{equation}   
where $R_{0}$ and $\tau$ are given by \eqref{RS} and \eqref{tau} and
where $\be(0) \colon \cH(S) \to \cY$ is the evaluation-at-zero map:
$$
    \be(0) \colon f(z) \mapsto f(0).
$$
Then $\bU_{S}$ is coisometric, and one recovers $S(z)$ as the characteristic
function of $\bU_{S}$:
\begin{equation}   \label{realization}
  S(z) = D_{S} + z C_{S} (I - zA_{S})^{-1} B_{S}.
\end{equation}
\end{enumerate}
\end{theorem}

\section{de Branges-Rovnyak spaces and Schur-class interpolation}  
\label{S:interpolation}

This section will show how de Branges-Rovnyak spaces appear in a 
natural way in the context of Schur-class interpolation theory.
The main idea comes from the work of Katsneslon, Kheifets and 
Yuditskii \cite{kat1, kat, kky, kh, khyu}, and is closely connected with
Potapov's method of Fundamental Matrix Inequalities 
\cite{kat, khyu, kov1, kov2, kov3, kovpotap, potap1}.

\smallskip

The starting point is a relatively simple left-tangential operator-valued version of the 
classical Nevanlinna-Pick problem which consists of the following: {\em Given $n$ distinct points 
$z_1,\ldots,z_n\in\D$  and given  vectors $E_1,\ldots,E_n\in\cY$ and $N_1,\ldots, 
N_n\in\cU$,
find a Schur-class function  $S\in\cS(\cU,\cY)$  (if such exists) such that 
\begin{equation}
S(z_i)^*E_i=N_i\quad\mbox{for}\quad i=1,\ldots,n.
\label{5.1}
\end{equation}}
In what follows, $E_i^*$ and $N_i^*$ will be viewed as elements of $\cL(\cY,\C)$ and 
$\cL(\cU,\C)$, respectively. Upon multiplying both parts in \eqref{5.1} by $k_{\rm 
Sz}(\cdot, z_i)$
and making use of formula \eqref{ts*} one concludes that \eqref{5.1} can be written 
equivalently 
in terms of the Toeplitz operator $T_S$ as
$$
T_S^*: \; E_i(1-z\overline{z}_i)^{-1}\mapsto 
N_i(1-z\overline{z}_i)^{-1}\quad\mbox{for}\quad i=1,\ldots,n
$$
or equivalently, as the single condition 
\begin{equation}
T_S^*: \; \sum_{i=1}^n E_i(1-z\overline{z}_i)^{-1}x_i\mapsto 
\sum_{i=1}^nN_i(1-z\overline{z}_i)^{-1}x_i
\label{5.2}   
\end{equation}
holding for all $x_1,\ldots,x_n\in\C$.  Introduce the operators 
\begin{equation}
T=\begin{bmatrix}\overline{z}_1 & & 0\\ &\ddots & \\ 0 && 
\overline{z}_n\end{bmatrix},\quad
E=\begin{bmatrix}E_1 & \ldots & E_n\end{bmatrix},\quad
N=\begin{bmatrix}N_1 & \ldots & N_n\end{bmatrix}
\label{5.3}   
\end{equation}
and two observability operators $\cO_{E,T}: \, \C^n\to H^2(\cY)$ and $\cO_{N,T}: \, 
\C^n\to H^2(\cU)$ 
defined as
\begin{equation}
\cO_{E,T}: \, x\mapsto E(I-zT)^{-1}x\quad\mbox{and}\quad
\cO_{N,T}: \, x\mapsto N(I-zT)^{-1}x.
\label{5.4}
\end{equation}
It is readily seen from \eqref{5.3}, \eqref{5.4} that condition \eqref{5.2} can be 
equivalently written in 
the operator form as 
\begin{equation}
T_S^*\cO_{E,T}=\cO_{N,T}.
\label{5.5}
\end{equation}
In general  the pair of operators $(E,T)$ (where say 
$E \in \cL(\cX, \cY)$ and $T \in \cL(\cX)$) is  said to be {\em output stable} if 
the associated observability operator as in \eqref{5.4} maps $\cX$ 
into $H^{2}(\cY)$.
This discussion suggests the more general interpolation problem: 

\medskip

{\bf IP:} {\em Given Hilbert space operators 
$T\in\cL(\cX)$, $E\in\cL(\cX, \cY)$ and $N\in\cL(\cX,\cU)$ such that the pairs 
$(E,T)$ and $(N,T)$ are output-stable, find a Schur-class function  $S\in\cS(\cU,\cY)$ 
subject to interpolation condition \eqref{5.5}.}

\medskip

The Nevanlinna-Pick problem recalled above is a particular case of the problem 
{\bf IP} corresponding to $\cX=\C^n$ and to the special choice \eqref{5.3} of the 
operators 
$T$, $E$ and $N$. 

\smallskip

Observe that for an output-stable pair $(E,T)$ and for any $x\in\cX$,
$$
\|\cO_{E,T}x\|^2_{H^2(\cY)}-\|\cO_{E,T}Tx\|^2_{H^2(\cY)}=
\sum_{k=0}^\infty\|ET^kx\|^2_{\cY}-\sum_{k=0}^\infty\|ET^{k+1}x\|^2_{\cY}=\|Ex\|^2_{\cY}
$$
and similarly,
$$
\|\cO_{N,T}x\|^2_{H^2(\cU)}-\|\cO_{N,T}Tx\|^2_{H^2(\cU)}=\|Nx\|^2_{\cU}. 
$$
Define
\begin{equation} 
P:=\cO_{E,T}^*\cO_{E,T}-\cO_{N,T}^*\cO_{N,T}. 
\label{5.6}   
\end{equation}
Then it  follows that
\begin{align*}
\langle Px,\, x\rangle_{\cX}-\langle PTx,\, 
Tx\rangle_{\cX}=&\|\cO_{E,T}x\|^2_{H^2(\cY)}-\|\cO_{N,T}x\|^2_{H^2(\cU)}\\
&-\|\cO_{E,T}Tx\|^2_{H^2(\cY)}+\|\cO_{N,T}Tx\|^2_{H^2(\cU)}\\
=&\|Ex\|^2_{\cY}-\|Nx\|^2_{\cU}\quad\mbox{for all}\quad x\in\cX
\end{align*}
which can be written in operator form as 
\begin{equation}  
P-TPT^*=EE^*-NN^*.
\label{5.7}  
\end{equation}
The operator $P$ defined above from interpolation data is called {\em the Pick operator} 
of the
problem {\bf IP}. 
Observe, that in case \eqref{5.3} of the tangential Nevanlinna-Pick problem, $P$ admits 
the explicit
matrix formula
\begin{equation}
P=\left[\frac{\langle E_i, \, E_j\rangle_{\cY}-\langle N_i, \, N_j\rangle_{\cU}}{1-z_i
\overline{z}_j}\right]_{i,j=1}^n.
\label{5.8}
\end{equation}
If the problem {\bf IP} has a solution (say, $S\in\cS(\cU,\cY)$), then equality 
\eqref{5.5} holds for a 
contraction operator $T_S^*$ and therefore, 
$$
\|\cO_{N,T}x\|^2_{H^2(\cU)}=\|T_S^*\cO_{E,T}x\|^2_{H^2(\cU)}\le 
\|\cO_{E,T}x\|^2_{H^2(\cY)}\quad\mbox{for 
all}\quad x\in\cX
$$
which simply means that the Pick operator \eqref{5.6} is positive semidefinite. 
The necessity part of the next result follows from this discussion.

\begin{theorem}
The problem {\bf IP} has a solution if and only if its Pick matrix is positive 
semidefinite:
\begin{equation}
P:=\cO_{E,T}^*\cO_{E,T}-\cO_{N,T}^*\cO_{N,T}\ge 0.
\label{5.9}     
\end{equation}
\label{T:5.1}
\end{theorem}
\begin{remark}
{\rm Taking adjoints in \eqref{5.5} gives $\cO_{E,T}^*T_S=\cO_{N,T}^*$ where operators on 
both sides
map $H^2(\cU)$ into $H^2(\cY)$. Upon restricting this operator equality to the coefficient 
space $\cU$
(that is, to the space of constant functions in $H^2(\cU)$) one gets 
\begin{equation}
\cO_{E,T}^* M_S\vert_{\cU}=\cO_{N,T}^*\vert_{\cU}=N^*.
\label{5.9a}
\end{equation} 
The latter condition is a consequence of \eqref{5.5}. However, it can be equivalently used 
in the 
formulation of the {\bf IP} for the following reason:} if the pair $(E,T)$ is output 
stable and 
equality \eqref{5.9a} holds for a Schur-class function $S\in\cS(\cU,\cY)$, then the pair 
$(N,T)$ is also output stable (so that the observability operator $\cO_{N,T}$ maps $\cX$ 
into 
$H^2(\cU)$) and equality \eqref{5.5} holds.
\label{R:equ}
\end{remark}
At this point de Branges-Rovnyak spaces come into play, With any Schur-class 
function $S\in\cS(\cU,\cY)$, one can associate the linear map $F^S: \, \cX\to H^2(\cY)$ by 
the formula
\begin{equation}
F^S: \, x\mapsto \left(\cO_{E,T}-T_S\cO_{N,T}\right)x.
\label{5.10}
\end{equation}
If $S$ satisfies condition \eqref{5.5}, then 
$$
F^Sx=\left(\cO_{E,T}-T_ST_S^*\cO_{E,T}\right)x=\left(I-T_ST_S^*\right)\cO_{E,T}x
$$
and therefore $F^Sx$ belongs to $\cH(S)$ by characterization \eqref{def2HS}. Moreover,
 \begin{align*}
\|F^Sx\|^2_{\cH(S)}&=\left\langle (I-T_ST_S^*)\cO_{E,T}, \, 
\cO_{E,T}\right\rangle_{H^2(\cY)}\\
& = \langle (\cO_{E,T}^{*} \cO_{E,T} -
          \cO_{N,T}^{*}\cO_{N,T})x, x \rangle_{\cX}
           = \langle Px, x \rangle_{\cX} = \| P^{\frac{1}{2}} x\|_{\cX}^{2}.
\end{align*}
It has been shown that under the assumption \eqref{5.10},  
a function $S\in\cS(\cU,\cY)$ is a solution to the problem {\bf IP} only if the linear 
transformation 
\eqref{5.10} maps $\cX$ into $\cH(S)$ with equality $\|F^Sx\|^2_{\cH(S)}= \| 
P^{\frac{1}{2}} x\|_{\cX}$
for every $x\in\cX$. The converse ("if") statement was established in \cite{kky}. This and 
several other 
characterizations of solutions to the problem {\bf IP} are presented in the next theorem. 
In some statements, the function $S$ will not be assumed to be in the 
Schur class; consequently the notation 
$M_S: \, 
f\mapsto 
Sf$ rather than $T_S$ will be used for the operator of multiplication by $S$. 

\begin{theorem}
Assume that condition \eqref{5.9} is satisfied and let $F^S$ be defined
as in \eqref{5.10} (with $M_S$ instead of $T_S$) for a function $S: \, \D\to 
\cL(\cU,\cY)$.
The following are equivalent:
\begin{enumerate}
\item $S$ is a solution of  the problem {\bf IP}.
\item $S\in\cS(\cU,\cY)$ and the function $F^Sx$ belongs to 
$\cH(S)$ and satisfies
\begin{equation}
\|F^Sx\|_{\cH(S)}= \|P^{\frac{1}{2}}x\|_\cX \quad\mbox{for
every}\quad x\in\cX.
\label{5.12}
\end{equation}
\item $S\in\cS(\cU,\cY)$ and
the function $F^Sx$ belongs $\cH(S)$ and satisfies
\begin{equation}
\|F^Sx\|_{\cH(S)}\le \|P^{\frac{1}{2}}x\|_\cX \quad\mbox{for
every}\quad x\in\cX.
\label{5.13}
\end{equation}
\item The following kernel is positive in $\D\times\D$:
\begin{equation}
\mathbb K_S(z,\zeta)=
\begin{bmatrix}P & (I_\cX-\overline{\zeta}T^*)^{-1}(E^*-N^*S(\zeta)^*) \\
(E-S(z)N)(I_{\cX}-zT)^{-1}
 & {\displaystyle\frac{I_\cY -S(z)S(\zeta)^*}{1-z\overline{\zeta}}}\end{bmatrix}\succeq 0.
\label{5.14}  
\end{equation}
\item $F^S$ maps $\cX$ into $H^2(\cY)$ and the operator
\begin{equation}
{\bf P}:=\begin{bmatrix} P & (F^S)^* \\ F^S & I-M_SM_S^*\end{bmatrix}
\colon \; \begin{bmatrix}{\mathcal X} \\ H^2(\cY)\end{bmatrix}
\to  \begin{bmatrix}{\mathcal X} \\ H^2(\cY)\end{bmatrix}
\label{5.15}
\end{equation}  
is positive semidefinite.
\end{enumerate}
\label{T:5.3}
\end{theorem}
\begin{proof} A brief sketch will be given. Implication $(1)\Rightarrow(2)$ was 
demonstrated above.
Implication $(2)\Rightarrow(3)$ is trivial. Implication $(3) 
\Rightarrow (4)$ follows from Proposition \ref{bourb}. Implication 
$(4) \Rightarrow (5)$ follow from the identity
        $$
        \left\langle {\bf P}f, \; f\right\rangle_{{\mathcal
        X}\oplus H^2(\cY)}
        =\sum_{j,\ell=1}^r \left\langle{\mathbb K_S}(z_j,z_\ell)
        \left[\begin{array}{c} x_\ell \\ y_\ell \end{array}\right], \;
        \left[\begin{array}{c} x_j \\ y_j \end{array}\right]
        \right\rangle_{{\mathcal X}\oplus\cY}
        $$
holding for every vector $f\in{\mathcal 
        X}\oplus H^2(\cY)$ of the form
        $$
        f=\sum_{j=1}^r\left[\begin{array}{c} x_j\\
        k_{\rm Sz}(\cdot \, , \, z_j)y_j \end{array}\right]\qquad
        (x_j\in{\mathcal X}, \; y_j\in\cY, \; z_j\in\D).
        $$
For implication $(5) \Rightarrow (1)$, first observe that since 
$I-M_SM_S^*$ is positive semidefinite (equivalently,
$M_S$ is a contraction) then $S\in\cS(\cU,\cY)$ and $M_S=T_S$.
By  definitions \eqref{5.6} and \eqref{5.10},
$$
{\bf P}=\begin{bmatrix} {\mathcal O}_{E,T}^*{\mathcal O}_{E,T}-
{\mathcal O}_{N,T}^*{\mathcal O}_{N,T} & {\mathcal O}_{E,T}^*-
{\mathcal O}_{N,T}^*T_S^*  \\ {\mathcal O}_{E,T}-T_S{\mathcal O}_{N,T} &
       I-T_ST_S^*\end{bmatrix}\ge 0.
       $$
       By the standard Schur complement argument, the latter
       inequality is equivalent to
       $$
       \widehat {\bf P}:=\begin{bmatrix}
       I_{H^2(\cU)} & {\mathcal O}_{N,T} & T_S^*\\
       {\mathcal O}_{N,T}^*
        & {\mathcal O}_{E,T}^*{\mathcal O}_{E,T} & {\mathcal O}_{E,T}^*
       \\ T_S & {\mathcal O}_{E, \bT} & I_{H^2(\cY)}\end{bmatrix}\ge 0,
       $$
       since ${\bf P}$ is the Schur complement of the block
       $I_{H^2(\cU}$ 
       in $\widehat {\bf P}$. On the other hand, the latter inequality
       holds if and only if the Schur
       complement of the block $I_{H^2(\cY)}$ in $\widehat {\bf P}$
       is positive semidefinite:
$$
       \begin{bmatrix} I_{H^2(\cU)} & {\mathcal O}_{N,T}\\
       {\mathcal O}_{N,T}^* & {\mathcal O}_{E,T}^*{\mathcal O}_{E,T}
       \end{bmatrix}-\begin{bmatrix} T_S^* \\ {\mathcal O}_{E,T}^*
       \end{bmatrix}\begin{bmatrix} T_S& {\mathcal O}_{E,T}\end{bmatrix}\ge 0.
$$
       One can write the latter inequality as
       $$
       \begin{bmatrix} I_{H^2(\cU)}-M_S^*M_S & {\mathcal O}_{N,T}-T_S^*{\mathcal 
O}_{E,T}\\ 
{\mathcal O}_{N,T}^*-{\mathcal O}_{E,T}^*T_S & 0\end{bmatrix}\ge
       0
       $$
       and arrive at  ${\mathcal O}_{E,T}^*T_S={\mathcal O}_{N,T}$
       which means that $S$ is a solution of {\bf IP}.
\end{proof}
It can be shown that Theorem \ref{T:5.3} holds in a more general setting of contractive 
multipliers from one reproducing kernel Hilbert space into another \cite{bol}.

\subsection{V.P. Potapov's method of Fundamental Matrix Inequalities} Theorem \ref{T:5.3}
originates in the approach suggested by V. P. Potapov in early 1970s 
and developed later by his collaborators and followers. The method consisted of three 
parts:
given an interpolation problem, 
\begin{enumerate}
\item establish the solvability criterion in terms of the Pick operator  $P$ of the 
problem
 and establish the identity (the "fundamental identity" in Potapov's terminology) 
satisfied by this $P$;
\item characterize all solutions $S$ to the problem in terms of the "fundamental matrix 
inequality" 
${\bf K}=\sbm{P & \star \\ \star &\star}\ge 0$ where ${\bf K}$ is certain structured 
matrix depending on the 
unknown 
function $S$ and having $P$ as a diagonal block;
\item describe all solutions $S$ of the inequality ${\bf K}\ge 0$ using factorization 
methods.
\end{enumerate}
One of the main reasons to develop this method was that in the 
completely indeterminate case (where $P$ is strictly
positive definite), the operator-valued problem can be  settled in 
much the same way as in the scalar-valued case.
The method was tested on a number of classical interpolation problems \cite{dub1, kat1, 
kov1, kov2, kovpotap, 
potap1} and then was largely unified and extended in \cite{kky} (see 
also \cite{khyu}).  Problem {\bf IP} can be used to illustrate 
Potapov's method as follows. 
The solvability criterion is given in \eqref{5.9} in terms of $P$ which 
satisfies the "fundamental identity" \eqref{5.7}. The next step is presented in the 
theorem below.

 \begin{theorem}
Let $P$ be defined as in \eqref{5.9}.
A function $S: \, \D\mapsto \cL(\cU,\cY)$ is a solution to the problem {\bf IP} if and 
only if
it is analytic on $\D$ and the following matrix is positive semidefinite for all $z\in\D$:
\begin{equation}
\begin{bmatrix}P & (I_\cX-\overline{z}T^*)^{-1}(E^*-N^*S(z)^*) \\
(E-S(z)N)(I_{\cX}-zT)^{-1}
 & {\displaystyle\frac{I_\cY -S(z)S(z)^*}{1-|z|^2}}\end{bmatrix}\ge 0.
\label{5.16}
\end{equation}
\label{T:5.2}  
\end{theorem}
The proof (for the case where $\cX$, $\cY$ and $\cU$ are all finite dimensional) can be 
found in
\cite[Section 3]{bd1}. The "if" part is a fairly straightforward consequence of the 
Schwarz-Pick
inequality (of course, this part  follows also from Theorem \ref{T:5.3}, since the matrix 
in \eqref{5.16}
is nothing else but $\mathbb K_S(z,z)$ and therefore condition \eqref{5.14} is stronger 
than 
\eqref{5.16}). The "only if" part is much trickier. Interpolation conditions are derived
from \eqref{5.16} using a special transformation of the latter inequality suggested first 
in 
\cite{kky} (see also \cite{kat} for a related survey). Further developments showed that 
it is much more convenient to work with positive kernels rather than positive semidefinite 
matrices.
Besides, as one can see from Theorem \ref{T:5.3}, the ``kernel" setting makes connections 
between 
Nevanlinna-Pick type interpolation problems and de Branges-Rovnyak spaces more 
transparent. 

\subsection{The analytic Abstract Interpolation Problem} The very formulation of the 
problem 
{\bf IP} requires that the observability operators $\cO_{E,T}$ and $\cO_{N,T}$ be bounded 
from $\cX$ into 
$H^2(\cY)$ and $H^2(\cU)$ respectively. Besides, the special form \eqref{5.6} of the 
operator $P$ is essential  
for proving implication $(5) \Rightarrow(1)$ in Theorem \ref{T:5.3}. However, upon close 
inspection, one can 
see 
that the equivalences $(3) \Leftrightarrow(4) \Leftrightarrow(5)$ in Theorem \ref{T:5.3} 
survive under
weaker assumptions that $P$ is any positive semidefinite operator on
$\cX$ and that 
\begin{itemize}
\item[(a)] {\em The function $\begin{bmatrix}E \\  N\end{bmatrix}(I-zT)^{-1}x$ is 
holomorphic on 
$\D$ for each $x\in\cX$}. 
\end{itemize}
For reasons explained below, it should also be  required that 
\begin{itemize} 
\item[(b)] {\em $P$ is a positive semidefinite solution to the Stein equation 
\eqref{5.7}},
\end{itemize}
and formulate the Abstract Interpolation Problem as follows:

\medskip

{\bf AIP:}  {\em Given the data $\{E,N,T,P\}$ subject to assumptions {\rm (a), (b)}, find 
all 
$S\in\cS(\cU,\cY)$ 
such that for every $x\in\cX$, the function 
\begin{equation}
(F^Sx)(z)=(E-S(z)N)(I_\cX-zT)^{-1}x
\label{5.17}
\end{equation}
belongs to the de Branges-Rovnyak space $\cH(S)$ and satisfies the norm
constraint $\|F^Sx\|_{\cH(S)}\le \|P^{\frac{1}{2}}x\|_\cX$.}

\medskip

The latter problem is a left-tangential adaptation of the more general bi-tangential 
Abstract Interpolation Problem
formulated in \cite{kky} (see also \cite{kh} for an overview)
in terms of a more elaborate two-component 
version of the de Branges-Rovnyak space 
$\widetilde{\mathcal D}(S)$ (a good reference for the formulation of 
this two-component space is \cite{NV1} as well as the survey article 
companion to this one \cite{bbsur1}). The present survey does not 
treat this more general interpolation problem. 

The next result can be arrived at via a 
careful inspection of the proof of Theorem \ref{T:5.3}.

\begin{theorem}
Let $P$, $T$, $E$ and $N$ satisfy assumptions {\rm (a), (b)}. Then a function 
$S: \, \D\to\cL(\cU,\cY)$ is a solution of the  {\bf AIP} if and only if 
the kernel ${\mathbb K}_S(z,\zeta)$ of the form \eqref{5.14} is positive on $\D\times \D$.
\label{T:5.4}
\end{theorem}
An important example of a concrete interpolation problem which is a particular case of the 
problem 
{\bf AIP} but not of the {\bf IP} is the boundary interpolation problem \cite{bk6}.

\subsection{Parametrization of the solution set}
The third step of the Potapov method is to describe all functions 
$S$ such that the matrix \eqref{5.16} is positive semidefinite or, equivalently, such that 
the kernel
\eqref{5.14} is positive on $\D\times\D$. This was first done for the case where the Pick 
operator $P$
is strictly positive definite (in early developments, all the problems were matrix-valued 
and with finitely 
many 
interpolation conditions, so $P$ was a matrix which was assumed be positive definite). If 
$P$ is strictly 
positive 
definite, then it follows from factorization
$$
\mathbb K_S(z,\zeta)=\begin{bmatrix}I & 0 \\ F^S(z)P^{-1} & I \end{bmatrix}
\begin{bmatrix}P & 0 \\ 0 & K_S(z,\zeta)-F^S(z)P^{-1}F^S(\zeta)^*\end{bmatrix}
\begin{bmatrix}I & P^{-1}F^S(\zeta)^* \\ 0 & I\end{bmatrix}
$$
that \eqref{5.14} holds if and only if the kernel 
$$
\widetilde{K}_S(z,\zeta)= K_S(z,\zeta)-F^S(z)P^{-1}F^S(\zeta)^*
$$
is positive on $\D\times\D$. Using the definitions \eqref{deBRkernel} and \eqref{5.17} of 
$K_S$ and $F^S$ and 
making 
use of the operators 
\begin{equation}
J=\begin{bmatrix} I_\cY & 0 \\ 0 & -I_\cU\end{bmatrix}\quad \mbox{and}\quad 
C=\begin{bmatrix}E \\ 
N\end{bmatrix},
\label{5.18}
\end{equation}
one can represent the kernel $\widetilde{K}_S$ as 
\begin{align}
\widetilde{K}_S(z,\zeta)=&\frac{I_\cY-S(z)S(\zeta)^*}{1-z\overline{\zeta}}\label{5.19}\\
&-(E-S(z)N)(I-zT)^{-1}P^{-1}(I-\overline{\zeta}T^*)^{-1}(E^*-N^*S(\zeta)^*)\notag\\
=&\begin{bmatrix}I & -S(z)\end{bmatrix}\left\{\frac{J}{1-z\overline{\zeta}}-
C(I-zT)^{-1}P^{-1}(I-\overline{\zeta}T^*)^{-1}C^*\right\}\begin{bmatrix}I \\ 
-S(\zeta)^*\end{bmatrix}.
\notag
\end{align}
The crucial step is to find a function $\Theta: \, \D\to\cL(\cY\oplus\cU)$ such that 
\begin{equation}
\frac{J-\Theta(z)J\Theta(\zeta)^*}{1-z\overline{\zeta}}=
C(I-zT)^{-1}P^{-1}(I-\overline{\zeta}T^*)^{-1}C^*\qquad (z,\zeta\in\D).
\label{5.20} 
\end{equation}
If ${\rm spec} T\cap{\mathbb T}\neq \mathbb T$, i.e., if  there exists a boundary point 
$\mu\in\mathbb T$ such 
that 
$(\mu I-T^*)^{-1}\in\cL(\cX)$, one may try to find a $\Theta$ normalized by 
$\Theta(\mu)=I_{\cY\oplus\cU}$. 
Letting $\zeta=\mu$ in \eqref{5.20} gives
$$
J-\Theta(z)J=(1-z\overline{\mu})C(I-zT)^{-1}P^{-1}(I-\overline{\mu}T^*)^{-1}C^*
$$
which then implies
\begin{align*}
\Theta(z)&=I-(1-z\overline{\mu})C(I-zT)^{-1}P^{-1}(I-\overline{\mu}T^*)^{-1}C^*J\\
&=I+(z-\mu)C(I-zT)^{-1}P^{-1}(\mu I-T^*)^{-1}C^*J,
\end{align*}
and eventually, on account of \eqref{5.18},
\begin{align} 
\Theta(z)&=\begin{bmatrix}\Theta_{11}(z) & \Theta_{12}(z)\\
\Theta_{21}(z) & \Theta_{22}(z)\end{bmatrix}\notag\\
&=I+(z-\mu)\begin{bmatrix}E \\ N\end{bmatrix}(I-zT)^{-1}P^{-1}(\mu 
I-T^*)^{-1}\begin{bmatrix}E^* & 
-N^*\end{bmatrix}.
\label{5.21}
\end{align}
The accomplishment so far is a function satisfying \eqref{5.20} for every $z\in\D$ and a 
fixed $\zeta=\mu\in\mathbb 
T$.  A straightforward calculation based solely on the Stein identity \eqref{5.7} shows 
that the function 
\eqref{5.21}
actually satisfies the identity \eqref{5.20} for all $z,\zeta\in\D$. Moreover, another 
calculation (again 
based on 
the identity  \eqref{5.7} only) shows that 
\begin{equation}
\frac{J-\Theta(z)^*J\Theta(\zeta)}{1-\bar z\zeta}=\widetilde{C}(I-\bar z 
T^*)^{-1}P(I-\zeta 
T)^{-1}\widetilde{C}^*
\label{5.21a} 
\end{equation}
where
$$
\widetilde{C}=JC(I-\mu T)^{-1}P^{-1}(\mu I-T^*).
$$
Formulas \eqref{5.20} and \eqref{5.21} show that the function $\Theta$ is 
$J$-bicontractive, i.e., that 
\begin{equation}
\Theta(z)J\Theta(z)^*\le J\quad\mbox{and}\quad \Theta(z)^*J\Theta(z)\le J\quad \mbox{for 
all}\quad z\in\D.
\label{5.21b}
\end{equation}
Another method of constructing a $J$-contractive $\Theta$ subject to the identity 
\eqref{5.20} is based on the 
Kre\u{\i}n 
space arguments.
\begin{lemma}
Let $P$ be a strictly positive solution to the Stein equation 
\begin{equation}
P-T^*PT=E^*E-N^*N=C^*JC.
\label{5.21c}   
\end{equation}
Then there exists an injective operator $\sbm {B \\ D}: \, \cY\oplus\cU\to 
\cX\oplus\cY\oplus\cU$ such that 
\begin{align}
\begin{bmatrix} T & B \\ C & D \end{bmatrix}\begin{bmatrix}P^{-1} & 0  \\ 0 & 
J\end{bmatrix}
\begin{bmatrix} T^* & C^* \\ B^* & D^* \end{bmatrix}&=\begin{bmatrix}P^{-1} & 0  \\ 0 & 
J\end{bmatrix},
\label{5.22}\\
\begin{bmatrix} T^* & C^* \\ B^* & D^* \end{bmatrix}\begin{bmatrix}P & 0  \\ 0 & 
J\end{bmatrix}
\begin{bmatrix} T & B \\ C & D \end{bmatrix}&=\begin{bmatrix}P & 0  \\ 0 & J\end{bmatrix}.
\label{5.23}
\end{align}
\label{L:5.5}
\end{lemma}
\begin{proof} It is seen from the Stein identity \eqref{5.21c} that ${\mathcal G } : 
=\operatorname{Ran}\sbm{ T \\ 
C}$ 
is a 
uniformly positive subspace of the 
Kre\u{\i}n space $\mathcal K=\cX \oplus\cY \oplus \cU$ with inner product induced by the 
operator  $\sbm{P & 0 
\\ 
0 
& J}$. The Kre\u{\i}n-space orthogonal projection of $\cK$  onto ${\mathcal G}$
is given by ${\mathcal P}_{{\mathcal G}} = \sbm{T \\ C}\sbm{T \\ C}^{[*]}$
where the Kre\u{\i}n-space adjoint of $\sbm{T \\ C}$ is given by
$$
\begin{bmatrix} T \\ C\end{bmatrix}^{[*]} =  
P^{-1} \begin{bmatrix} T^{*} & C^* \end{bmatrix}
\begin{bmatrix} P & 0 \\ 0 & J \end{bmatrix}.
$$
Therefore, the Kre\u{\i}n-space orthogonal projection ${\mathcal 
P}_{{\mathcal G}^{[\perp ]}}$ equals
\begin{equation}  \label{5.24}
{\mathcal P}_{{\mathcal G}^{[\perp]}} = I_{\cK}
-{\mathcal P}_{{\mathcal G}} = I - \begin{bmatrix} T \\ C \end{bmatrix}
P^{-1} \begin{bmatrix} T^{*} & C^*\end{bmatrix}
\begin{bmatrix} P & 0 \\ 0 & J \end{bmatrix}.
\end{equation}
On the other hand, since $\cG$ is a uniformly positive subspace of $\cK$, its orthogonal 
complement ${\mathcal 
G}^{[\perp]}$ is also a Kre\u{\i}n space in inner product inherited from $\cK$ with 
inertia equal to that of 
$J$ 
on 
$\cY \oplus \cU$. Therefore there is an injective isometry 
$$
\begin{bmatrix} B \\ D \end{bmatrix}: \, \left( \begin{bmatrix} \cY 
\\ \cU \end{bmatrix}, J\right) \to \mathcal K\quad\mbox{such that}\quad {\mathcal 
P}_{{\mathcal G}^{[\perp]}} 
= 
\begin{bmatrix}B \\ D\end{bmatrix}\begin{bmatrix} B \\
D\end{bmatrix}^{[*]}.
$$
Since  $\begin{bmatrix} B \\ D \end{bmatrix}^{[*]} =J \begin{bmatrix} B^{*} & D^{*} 
\end{bmatrix}
      \begin{bmatrix} P & 0 \\ 0 & J \end{bmatrix}$, it follows that
\begin{equation}
\label{5.25}
{\mathcal P}_{{\mathcal G}^{[\perp]}} = \begin{bmatrix} B \\ D
\end{bmatrix} J\begin{bmatrix} B^{*} & D^{*} \end{bmatrix}
\begin{bmatrix} & 0 \\ 0 & J \end{bmatrix}.
\end{equation}
Multiplying the two expressions \eqref{5.24} and \eqref{5.25} for
$P_{{\mathcal P}^{[ \perp ]}}$ by $\sbm{P^{-1} & 0 \\ 0 & J}$ on the right 
and using the subsequent equation gives 
\eqref{5.22}. Equality \eqref{5.23} then follows from the injectivity of $\sbm{B\\ D}$.
\end{proof}
With the operators $B$ and $D$ subject to operator equalities 
\eqref{5.22}, \eqref{5.23} in hand, the next step is to let
\begin{equation}  \label{5.26}
\Theta(z)=D+zC(I-z T)^{-1}B
\end{equation}
and then the identity \eqref{5.20} follows from \eqref{5.22} whereas the identity 
$$
\frac{J-\Theta(z)^*J\Theta(\zeta)}{1-\bar z\zeta}=B^*(I-\bar z T^*)^{-1}P(I-\zeta T)^{-1}B
$$
is a consequence of \eqref{5.23}. The function $\Theta$ obtained this way also satisfies 
inequalities 
\eqref{5.21}.
\begin{remark}
{\rm If a solution $P$ to the Stein equation \eqref{5.7} is strictly positive definite. it 
then 
follows that the operator $T$ is strongly stable , which in turn implies that the function 
$\Theta$ is 
$J$-inner, i.e., that is,
the nontangential boundary values $\Theta(t)$ exist for almost all $t\in\mathbb T$ and are
$J$-unitary: $\Theta(t)J\Theta(t)^{*}= J$.}
\label{R:inner}
\end{remark}
\begin{theorem} 
Let $P$ be a strictly positive solution to the Stein equation \eqref{5.21c} and let 
$\Theta$ be a 
$J$-bicontractive function satisfying the identity \eqref{5.20}. Then a function $S: \, 
\D\to\cL(\cU,\cY)$ 
is a solution to the problem {\bf AIP} if and only if it is of the form 
\begin{equation}  \label{5.27}
S=(\Theta_{11}\cE+\Theta_{12})(\Theta_{21}\cE+\Theta_{22})^{-1}
\end{equation}
for some $\cE\in\cS(\cU,\cY)$.
\label{T:5.6}
\end{theorem}
\begin{proof}
Substituting the block decomposition $\Theta=\sbm{\Theta_{11} & \Theta_{12}\\ \Theta_{21} 
& \Theta_{22}}$
conformal with that of $J$ into inequalities \eqref{5.21b} gives in particular,
\begin{align*}
& \Theta_{21}(z)\Theta_{21}(z)^{*} -\Theta_{22}(z)\Theta_{22}(z)^{*} \le - I_{\cU}, \\    
& \Theta_{12}(z)^{*}\Theta_{12}(z) -\Theta_{22}(z)^{*}\Theta_{22}(z) \le -I_{\cU},       
\end{align*}
from which it follows that $\Theta_{22}(z)$ is invertible and that 
$\|\Theta_{22}^{-1}(z)\Theta_{21}(z)\|\le 
1$ for all $z\in\D$. Therefore, 
$$
\Theta_{21}(z)\cE(z)+\Theta_{22}(z)=\Theta_{22}(z)(\Theta_{22}^{-1}(z)\Theta_{21}(z)\cE(z)+I)
$$ 
is invertible for all $z\in\D$ and $\cE\in\cS(\cU,\cY)$ and thus the formula \eqref{5.27} 
makes sense.

\smallskip

One can now substitute \eqref{5.20} into \eqref{5.19} and conclude that the kernel $\mathbb 
K_S$ is positive 
on $\D\times\D$ if and only if 
\begin{equation}  \label{5.28}
\widetilde{K}_S(z,\zeta)=\begin{bmatrix}I & 
-S(z)\end{bmatrix}\frac{\Theta(z)J\Theta(\zeta)^*}{1-z\overline{\zeta}}
\begin{bmatrix}I \\ -S(\zeta)^*\end{bmatrix}\succeq 0.
\end{equation}
Set
\begin{equation}  \label{5.29}
u= \Theta_{11} - S\Theta_{21}, \qquad
v= S\Theta_{22}-\Theta_{12}.
\end{equation}
Then $\begin{bmatrix} u & - v \end{bmatrix} := \begin{bmatrix} I & -S\end{bmatrix}\Theta$,
        then it follows that 
$$
         \frac{ u(z) u(\zeta)^{*} - v(z) v(\zeta)^{*}}{1 - \langle z,
        \zeta \rangle} \succeq 0,
$$
which is equivalent, by Leech's theorem (see \cite[page 107]{RR}), to 
a factorization $v(z) = u(z) {\mathcal E}(z)$
for some ${\mathcal E} \in {\mathcal S}(\cU, \cY)$. On account of 
\eqref{5.29}, this in turn can be written 
as
$$
S\Theta_{22}-\Theta_{12}=(\Theta_{11}-S\Theta_{21})\cE.
$$
The latter can be rearranged as 
$S(\Theta_{21}\cE+\Theta)=\Theta_{11}\cE+\Theta_{12}$ which in turn, is equivalent to 
\eqref{5.27}.
\end{proof}
The formal obstacle to the use of the parametrization \eqref{5.27} in case $P\ge 0$ is 
singular 
is the presence of
$P^{-1}$ in the formula \eqref{5.21} for $\Theta$ (the inverse of $P$ also appears 
implicitly in formula 
\eqref{5.26} since the entries in this formula must satisfy equality \eqref{5.22}). A 
naive attempt
to overcome this difficulty (in case $\dim\cX<\infty$) would be to replace the inverse of 
$P$ by its 
Moore-Penrose pseudoinverse. Not for the general {\bf IP}, but at least for the 
left-tangential 
Nevanlinna-Pick 
problem 
\eqref{5.1}, the formula \eqref{5.21} produces all solutions to the problem if the 
parameter $\cE$ is taken in 
the 
form
$$
\cE(z)=U\begin{bmatrix} \widetilde{\cE}(z) & 0 \\ 0 & 
I_\nu\end{bmatrix}V,\quad\mbox{where}\quad 
\nu={\rm rank}(P+N^*N)-{\rm rank} P,
$$
where $U$ and $V$ are two matrices depending only on interpolation data and where 
$\widetilde{\cE}$ is an 
arbitrary 
Schur-class function. It was shown in \cite{dub2} for the matricial 
Schur-Carath\'eodory-Fej\'er problem and 
in 
\cite{bd1} for the general problem {\bf IP} (still with $\dim \cX<\infty$) that a similar 
result holds with 
an appropriate choice of the pseudoinverse of $P$ (not the Moore-Penrose in general) 
satisfying certain 
invariance 
relations.

\smallskip

In the case $\dim \cX=\infty$, this method does not seem to work beyond the situation 
where the compression of $P$ to the orthogonal complement of its kernel is 
strictly positive definite.  The following alternative approach handles the problem 
{\bf AIP} regardless of whether the operator $P$ is strictly positive 
definite or just positive semidefinite.

\subsection{Redheffer parametrization of the solution set} Once again 
the starting point is the Stein 
identity 
\eqref{5.7}
according to which
$$
\|P^{\frac{1}{2}}x\|_{\cX}^2+\|Nx\|_{\cU}^2=\|P^{\frac{1}{2}}Tx\|_{\cX}^2+\|Ex\|_{\cY}^2\quad\text{for
all}\quad x\in\cX.
$$
let $\cX_0=\overline{\rm Ran} P^{\frac{1}{2}}$; the conclusion from 
the Stein equality then is that 
there exists a well defined isometry
$V$ with domain ${\mathcal D}_V$ and range ${\mathcal R}_V$ equal to
$$
    {\mathcal {D}}_V=\overline{\rm Ran}\left[\begin{array}{c}
    P^{\frac{1}{2}} \\ N\end{array}\right]\subseteq \left[\begin{array}{c}
    \cX_0 \\  \cU\end{array}\right]\quad\mbox{and}\quad
    {\mathcal {R}}_V=\overline{\rm Ran}
    \left[\begin{array}{c}P^{\frac{1}{2}} T\\
    E\end{array}\right]\subseteq \left[\begin{array}{c}
    \cX_0 \\  \cY\end{array}\right],
$$
respectively, which is uniquely determined by the identity
 \begin{equation}\label{5.32}
    V \left[\begin{array}{c}P^{\frac{1}{2}} x \\
    Nx\end{array}\right]
   = \left[\begin{array}{c} P^{\frac{1}{2}} Tx\\
    Ex\end{array}\right]\quad\mbox{for all} \; \; x\in\cX.  
    \end{equation}
Let the defect spaces be defined by
    \begin{equation}\label{5.33}
    \Delta:=\left[\begin{array}{c}\cX_0 \\ \cU\end{array}\right]
    \ominus{\mathcal D}_V \quad{\rm and}\quad
    \Delta_*:=\left[\begin{array}{l}\cX_0\\
    \cY\end{array}\right]\ominus{\mathcal {R}}_V
    \end{equation}
    and let $\widetilde{\Delta}$ and $\widetilde{\Delta}_*$ denote
    isomorphic copies of $\Delta$ and $\Delta_*$, respectively, with
    unitary identification maps
    $$
    {\bf i}: \; \Delta\rightarrow \widetilde{\Delta}\quad\mbox{and}\quad
    {\bf i}_*: \; \Delta_*\rightarrow \widetilde{\Delta}_*.
    $$
With these identification maps let us define a unitary colligation matrix {\bf U} from
$\cD_V\oplus\Delta\oplus\widetilde\Delta_*=\cX\oplus\cU\oplus \widetilde\Delta_*$
onto $\cR_V\oplus\Delta_*\oplus \widetilde\Delta=\cX\oplus\cY\oplus \widetilde\Delta$
by
\begin{equation}\label{5.34}
{\bf U}=\begin{bmatrix}V & 0 & 0\\0 & 0 & {\mathbf i}_*^*\\0 & {\mathbf i} & 
0\end{bmatrix}
: \, \begin{bmatrix}\cD_V\\\Delta\\ \widetilde\Delta_*\end{bmatrix}
\to\begin{bmatrix}\cR_V\\\Delta_*\\ \widetilde\Delta\end{bmatrix},
\end{equation}
which will be also decomposed as
\begin{equation}\label{5.35}
{\bf U}=\begin{bmatrix}A&B_1&B_2\\C_1&D_{11}&D_{12}\\C_2&D_{21}&0\end{bmatrix}
:\, \begin{bmatrix}\cX_0\\\cU\\ \widetilde\Delta_*\end{bmatrix}
\to\begin{bmatrix}\cX_0\\\cY\\ \widetilde\Delta\end{bmatrix}.
\end{equation}
Write $\Sigma$ for the characteristic function associated with this
colligation ${\bf U}$, i.e.,
\begin{equation}\label{3.7}
\Sigma(z)=\begin{bmatrix}D_{11}&D_{12}\\D_{21}&0\end{bmatrix}
+z\begin{bmatrix}C_1\\C_2\end{bmatrix}(I-zA)^{-1}\begin{bmatrix}B_1&B_2\end{bmatrix}\quad(z\in\D),
\end{equation}
and decompose $\Sigma$ as
\begin{equation}\label{5.36}
\Sigma(z)=\begin{bmatrix}\Sigma_{11}(z)&\Sigma_{12}(z)\\\Sigma_{21}(z)&\Sigma_{22}(z)\end{bmatrix}
:\begin{bmatrix}\cU\\ \widetilde{\Delta}_*\end{bmatrix}\to\begin{bmatrix}\cY\\ 
\widetilde{\Delta}\end{bmatrix}.
\end{equation}
A straightforward calculation based on the fact that ${\bf U}$ is coisometric
gives
\begin{equation}
\frac{I-\Sigma(z)\Sigma(\zeta)^*}{1-z\overline{\zeta}}=
\begin{bmatrix}C_1\\C_2\end{bmatrix}(I-zA)^{-1}(I-\overline{\zeta}A^*)^{-1}
\begin{bmatrix}C_1^*&C_2^*\end{bmatrix},
\label{5.37}
\end{equation}
which implies in particular that $\Sigma$ belongs to the Schur class
$\cS(\cU\oplus\widetilde{\Delta}_*,\cY\oplus\widetilde{\Delta})$. 
\begin{theorem}
\label{T:5.7}
A function $S: \, \D\to\cL(\cU,\cY)$ is a solution of the problem {\bf AIP}
if and only if 
\begin{equation}
S={\mathcal R}_{\Sigma}[\cE]:=\Sigma_{11}+\Sigma_{12}(I-\cE\Sigma_{22})^{-1}\cE\Sigma_{21}
\label{5.38}   
\end{equation}
for some $\cE\in\cS(\widetilde{\Delta},\widetilde{\Delta}_*)$.
\end{theorem}
Note that by construction, $\Sigma_{22}(0)=0$ so that formula \eqref{5.38} makes sense for 
any
Schur-class function $\cE\in\cS(\widetilde{\Delta},\widetilde{\Delta}_*)$.
The proof of Theorem \ref{T:5.7} can be found in \cite{kky, kh}. 

In more detail, it is not hard to see that
if $\mathcal K$ is a Hilbert space containing $\cX$ and
\begin{equation}
U=\begin{bmatrix} A_{11} & A_{12}\\  A_{21} & A_{22}\end{bmatrix}: \, \begin{bmatrix}\cK 
\\ 
\cU\end{bmatrix}\to \begin{bmatrix}\cK \\ \cY\end{bmatrix}
\label{uniext}
\end{equation}
is a unitary operator such that 
$$
    U \left[\begin{array}{c}P^{\frac{1}{2}} x \\
    Nx\end{array}\right]
   = \left[\begin{array}{c} P^{\frac{1}{2}} Tx\\
    Ex\end{array}\right]\quad\mbox{for all} \; \; x\in\cX,
$$
(i.e., $U$ is a unitary extension of the isometry $V$ \eqref{5.32}), 
then the characteristic function 
$$
S(z)=A_{22}+zA_{12}(I-zA_{11})^{-1}A_{21}
$$
is a solution of the problem {\bf AIP}. A much less trivial fact (established in 
\cite{kky}) is that 
{\em 
any} solution to the problem {\bf AIP} arises in this way.  Then it remains to parametrize  
all
unitary extensions $U$ of the form \eqref{uniext} of the isometry \eqref{5.32} or (which 
is even better)
to parametrize the set of characteristic functions of all such extensions. The latter was 
done 
in \cite{arovgros1, arovgros2} via coupling of unitary colligations.

\smallskip

The conclusion of this section is a result needed  in the sequel; 
proofs can be found in \cite{bbth2}. 

\begin{proposition}
\label{P:5.10}
Let $\Sigma$ be the Schur-class function constructed as in \eqref{3.7} and decomposed as 
in 
\eqref{5.36}, and let $S$ be of the form \eqref{5.38} for a given 
$\cE\in\cS(\widetilde{\Delta},\widetilde{\Delta}_*)$.  Then the de Branges-Rovnyak kernels
$K_S$ and $K_\cE$ (see \eqref{deBRkernel}) are related as follows:
\begin{equation}
K_S(z,\zeta)= G(z)K_\cE(z,\zeta)G(\zeta)^*+\Gamma(z)\Gamma(\zeta)^*
\label{5.42}  
\end{equation}
where the functions $G$ and $\Gamma$ are defined on $\D$ in terms of $\Sigma$ by
\begin{equation}\label{5.43}
\begin{array}{rl}
G(z)\!\!\!&=\Sigma_{12}(z)(I-\cE(z)\Sigma_{22}(z))^{-1},\\[.2cm]
\Gamma(z)\!\!\!&=\left(C_1+G(z)\cE(z)C_2\right)(I-z A)^{-1}.
\end{array}
\end{equation}
Furthermore, the following equality  holds for all $z\in\D$:
\begin{equation}
\Gamma(z)P^{\frac{1}{2}}=(E-S(z)N(I-zT)^{-1}=F^S(z).
\label{5.44}
\end{equation}
\end{proposition}

\section{Interpolation in $\cH(S)$}

Interpolation problems in de Branges-Rovnyak spaces have not been considered 
until recently. The lack of interest in this topic can be explained by the fact that 
Hilbert space 
interpolation is well understood and no surprises are expected.  
However the results arising from the general Hilbert-space structure 
can be made much more explicit and concrete for this particular 
setting, as is discussed below.  Much of this Section is based on the 
papers of Ball, Bolotnikov, and ter Horst \cite{bbaip, bbth1, bbth2}.

\smallskip

Throughout this section, the Schur-class function $S\in\cS(\cU,\cY)$ is fixed. 
Hence the space $\cH(S)$ consists of $\cY$-valued functions and every function 
$f\in\cH(S)$ induces the multiplication operator $M_f: \, \C\to \cH(S)$ by the formula 
$M_f\alpha=f(z)\alpha$. In what follows, $A^*$ will denote the adjoint of $A: \, \cX\to 
\cH(S)\subset H^2(\cY)$ in the metric of  $H^2(\cY)$, and $A^{[*]}$ will denote he adjoint 
of $A$ in the metric of $\cH(S)$. Since these metrics are different (unless $S$ is 
inner), the adjoints $A^*$ and $A^{[*]}$ are not equal in general.

\smallskip

As in Section 3, the starting point is  a simple left-tangential
Nevanlinna-Pick problem: {\em Given $n$ distinct points 
$z_1,\ldots,z_n\in\D$  and given  vectors $E_1,\ldots,E_n\in\cY$ and given complex numbers
$y_1,\ldots,y_n$, find a function $f\in\cH(S)$ such that 
\begin{equation}
f(z_i)^*E_i=y_i\quad\mbox{for}\quad i=1,\ldots,n.
\label{6.1}
\end{equation}}
Making use of formula \eqref{ShieldsWallen} one may write the left hand side expression 
in \eqref{6.1} in terms of the adjoint operators $M_f^{*}: \, H^2(\cY)\to \C$  and 
 $M_f^{[*]}: \, \cH(S)\to \C$. Indeed,
\begin{equation}
f(z_i)^*E_i=M_f^{*}E_ik_{\rm 
Sz}(\cdot,z_i)=M_f^{*}(E_i(1-z\bar{z}_i)^{-1})\quad\mbox{for}\quad  
i=1,\ldots,n,
\label{6.2}
\end{equation}
and on the other hand,
\begin{align}
f(z_i)^*E_i&=M_f^{[*]}(K_S(\cdot,z_i)E_i)\notag\\\
&=M_f^{[*]} ((E_i-SS(z_i)^*E_i)k_{\rm Sz}(\cdot,z_i))\notag\\
&=M_f^{[*]} ((E_i-SN_i)k_{\rm Sz}(\cdot,z_i)) 
\quad\mbox{for}\quad  i=1,\ldots,n,\label{6.3}
\end{align}
where  $N_i:=S(z_i)^*E_i$ (recall that the function $S$ is given). Making use of 
matrices \eqref{5.3} and letting ${\bf y}=\begin{bmatrix} y_1 & \ldots & y_n\end{bmatrix}$,
one may rewrite $n$ conditions in \eqref{6.2} and \eqref{6.3} as
$$
M_f^{*}: \, E(I-z T)^{-1}x\mapsto {\bf y}x\quad\mbox{and}\quad
M_f^{[*]}: \; (E-S(z)N)(I-z T)^{-1}x\mapsto {\bf y}x.
$$
respectively, holding for all $x\in\C^n$. Next use the observability operators \eqref{5.4}
and the operator \eqref{5.10} to write the latter equalities in more compact form
$$
M_f^{*}\cO_{E,T}x={\bf y}x\quad\mbox{and}\quad
M_f^{[*]}F^Sx={\bf y}x
$$
or equivalently,
\begin{equation}
\cO_{E,T}^*f={\bf y}^*\quad\mbox{and}\quad (F^S)^{[*]}f={\bf y}^*.
\label{6.4}
\end{equation}
As will be shown below, the two latter conditions are equivalent in a much more general 
situation.
The first condition in \eqref{5.4} looks very much the same as that in \eqref{5.9a}, and 
this 
condition will be used  to formulate the problem {\bf IP} for functions in the space 
$\cH(S)$. 

\medskip

${\bf IP}_{\cH(S)}$: {\em Given a Schur-class function $S\in\cS(\cU,\cY)$, given 
an output-stable pair $(E,T)$ of operators $E\in\cL(\cX, \cY)$ and $T\in\cL(\cX)$, 
and given a functional ${\bf y}\in\cL(\cX,\C)$,  find a function  $f\in\cH(S)$ such that 
$\cO_{E,T}^*f={\bf y}^*$ and $\|f\|_{\cH(S)}\le 1$.}

\medskip

With the data set as above, one can introduce the operator $N\in\cL(\cX,\cU)$ via formula 
\eqref{5.9a}, 
that is, via its adjoint
\begin{equation}
N^*u=\cO_{E,T}^*(Su),\qquad u\in\cU.
\label{6.5}
\end{equation}
Since, $S$ is a Schur-class function, the pair $(N,T)$ is output-stable, and the operator 
$F^S$
given by \eqref{5.10} maps $\cX$ into $\cH(S)$. Since $S$ trivially 
solves the problem {\bf IP} with 
the current choice of $N$, inequality \eqref{5.9} holds by Theorem \ref{T:5.1} while 
equality 
\eqref{5.12} holds by  Theorem \ref{T:5.3}. Equality \eqref{5.12} can be written in the 
operator form 
as 
\begin{equation}
P=(F^S)^{[*]}F^S.
\label{6.6}
\end{equation}
Finally the equalities
 \begin{align*}
    \langle (F^S)^{[*]} f, \, x\rangle_{\cX}
    =\langle f, \, F^Sx\rangle_{\cH(S)}
    =&\langle f, \, (I-T_ST_S^*)\cO_{E,T}x\rangle_{\cH(S)}
    \notag \\
    =&\langle f, \, \cO_{E,T}x\rangle_{H^2(\cY)}=\langle
    \cO_{E,T}^*f, \, x\rangle_{\cX}\notag
    \end{align*}
hold for all $f\in\cH(S)$ and $x\in\cX$. Therefore,
$(F^S)^{[*]}=\cO_{E,T}^*|_{\cH(S)}$ and conditions \eqref{6.4} are equivalent in the 
general setting of the problem {\bf IP}.

\smallskip

As in the Schur-class setting, boundary interpolation problems cannot be embedded into the 
framework
of the problem {\bf IP}. To handle the boundary case, the stability assumption
on the pair $(E,T)$ need be relaxed. If the pair $(E,T)$ is not output-stable, we cannot 
use formula 
\eqref{6.5} to define $N$. Thus, the operator $N$ must be a part of interpolation 
data. Also the interpolation condition $\cO_{E,T}^*f={\bf y}^*$ cannot be formulated in 
this form since 
$\cO_{E,T}$ does not map $\cX$ into 
$H^2(\cY)$ and thus its range is not in $\cH(S)\subset H^2(\cY)$. 
Instead, one can assume that given $S(z), \, E, \, N, \, T$ are such that the operator 
$F^S$ defined as in \eqref{5.10} maps 
$\cX$ into  $\cH(S)$. Under this assumption one may use the second formula in \eqref{6.4} 
as the interpolation condition; on the other hand,  $P$ can be defined via formula 
\eqref{6.5}
instead of \eqref{5.6}. For the reasons already clear from what was seen in the previous 
section,
it makes sense to assume that the Stein identity \eqref{5.7} is in force. 
\begin{definition}
\label{D:6.1}
The data set 
\begin{equation}
{\mathcal D}=\{S, T, E, N, {\bf y}\}
\label{6.7}
\end{equation}
consisting of a Schur-class function $S\in\cS(\cU,\cY)$ and 
operators $T\in\cL(\cX)$, $E\in\cL(\cX,\cY)$, $N\in\cL(\cX,\cU)$, and
${\bf y}\in\cL(\cX,\C)$. is said to be ${\bf AIP}_{\cH(S)}$-admissible if:
\begin{enumerate}
\item The function $\begin{bmatrix}E \\  N\end{bmatrix}(I-zT)^{-1}x$ is holomorphic on
$\D$ for each $x\in\cX$.
\item The operator $F^S$ \eqref{5.10} maps $\cX$ into $\cH(S)$.
\item The operator $P:=(F^S)^{[*]}F^S$ satisfies the Stein equation \eqref{5.7}.
\end{enumerate}
\end{definition}
These preparations lead to the formulation of the problem ${\bf AIP}_{\cH(S)}$:

\medskip

{\em Given an ${\bf AIP}_{\cH(S)}$-admissible
data set \eqref{6.7}, find all $f\in\cH(S)$ such that
\begin{equation}
M_{F^S}^{[*]}f={\bf y}^*\quad\mbox{and}\quad\|f\|_{\cH(S)}\le 1.
\label{6.8}
\end{equation}}
The discussion preceding Definition \ref{D:6.1} shows that the problem ${\bf IP}_{\cH(S)}$
is a particular case of the problem ${\bf AIP}_{\cH(S)}$.

\subsection{The problem ${\bf AIP}_{\cH(S)}$ as a linear operator equation}
Consider the following operator interpolation problem with 
norm constraint: {\em Given Hilbert space operators $A\in\cL(\cY,\cX)$ and
$B\in\cL(\cU,\cX)$, find all $X\in\cL(\cU,\cY)$ that satisfy the conditions}
\begin{equation}
\label{6.9}
AX = B \quad\text{\em and}\quad \|X\| \le 1.
\end{equation}
According to the Douglas lemma \cite{Douglas65},
there is an $X\in\cL(\cU,\cY)$ satisfying \eqref{6.9} if and only
if $AA^*\ge BB^*$. If this is the case, then there exist (unique) contractions 
$X_1\in\cL(\cU,\overline{{\rm Ran}}A)$ and $X_2\in\cL(\cY,\overline{\rm Ran}A)$
such that
\begin{equation}\label{6.10}   
(AA^*)^{\frac{1}{2}}X_1=B,\quad
(AA^*)^{\frac{1}{2}}X_2=A, \quad {\rm Ker} X_1={\rm Ker} B,\quad
{\rm Ker} X_2={\rm Ker} A.
\end{equation}
The next  characterization of all operators $X$ subject to \eqref{6.9}
can be found in \cite{bbth2}.
\begin{lemma}
Assume $A A^{*}\geq BB^{*}$ and let $X\in\cL(\cU,\cY)$.
Then the following statements are equivalent:
\begin{enumerate}
\item $X$ satisfies conditions \eqref{6.9}.

\item The operator
\begin{equation}\label{6.11}
\begin{bmatrix}  I_{\cH_{1}} & B^{*} & X^{*} \\
B & A A^{*} & A \\ X & A^{*} & I_{\cH_{2}} \end{bmatrix}: \;
\begin{bmatrix}\cU\\\cX\\\cY\end{bmatrix}\to
\begin{bmatrix}\cU\\\cX\\\cY\end{bmatrix}
\end{equation}
is positive semidefinite.

\item $X$ is of the form
\begin{equation}\label{6.12}
X=X_2^*X_1+(I-X_2^*X_2)^{\frac{1}{2}}K (I-X_1^*X_1)^{\frac{1}{2}}
\end{equation}
where $X_1$ and $X_2$ are defined as in \eqref{6.10} and
where the parameter $K$ is an arbitrary contraction from
$\overline{\rm Ran}(I-X_1^*X_1)$ into $\overline{\rm Ran}(I-X_2^*X_2)$.
\end{enumerate}
Moreover, if $X$ satisfies \eqref{6.9}, then $X$ is unique if and only if
$X_1$ is isometric on $\cU$ or $X_2$ is isometric on $\cY$.
\label{L:6.2}
\end{lemma}
\begin{remark}  \label{R:6.3}
{\rm It follows from \eqref{6.12} that there is a unique $X$ subject to conditions
\eqref{6.9} if and only if $X_1$ is isometric on $\cU$ or $X_2$ is isometric on $\cY$.
Furthermore, since $X_2$ is a coisometry, it follows that
$(I-X_2^*X_2)^{\frac{1}{2}}$ is the
orthogonal projection onto $\cU\ominus{\rm Ker} A=\cU\ominus{\rm Ker}
X_1$. This implies that for each $K$ in \eqref{6.12} and  each $u\in\cU$, 
$$
\|Xu\|^2=\|X_2^*X_1 u\|^2+\|(I-X_2^*X_2)^{\frac{1}{2}}K
(I-X_1^*X_1)^{\frac{1}{2}}u\|^2,
$$
so that $X_2^*X_1$ is the minimal norm solution to the problem
\eqref{6.9} (see \cite{bbth2})}.
\end{remark}
Upon specifying the preceding discussion to the case where 
\begin{equation}
A=(F^S)^{[*]}:\, \cH(K_S)\to\cX,\quad B ={\bf y}^*\in\cX \cong\cL(\C,\cX),
\label{6.13}
\end{equation}
then it is readily seen that solutions $X \colon \C\to\cH(S)$ to problem \eqref{6.9} 
necessarily
have the form of a multiplication operator $M_f$ for some function
$f\in \cH(S)$.  This observation leads to the following solvability criterion.
\begin{theorem}
\label{T:6:4}
The problem ${\bf AIP}_{\cH(S)}$ has a solution if and only if
\begin{equation}
P\ge {\bf y}^*{\bf y},\quad\mbox{where}\quad
P:=(F^S)^{[*]}F^S.
\label{6.14}
\end{equation}
\end{theorem}
Assuming for simplicity that the operator $P$ is strictly positive definite,
it is readily  seen that
$$
X_1=P^{-\frac{1}{2}}{\bf y}^*\in\cX,\quad
X_2=P^{-\frac{1}{2}}(F^S)^{[*]}\in\cL(\cH(S),\cX)
$$
are the operators $X_1$ and $X_2$ from \eqref{6.10} after specialization to the case 
\eqref{6.13}.
The conclusion from \eqref{6.12} is that all solutions $f$ to the problem ${\bf 
AIP}_{\cH(S)}$ are 
given by the formula
\begin{equation}
f=F^SP^{-1}{\bf y}^*+\sqrt{1-\|P^{-\frac{1}{2}}{\bf y}^*\|^2}\cdot
(I-F^SP^{-1}(F^S)^{[*]})^{\frac{1}{2}}K
\label{6.15}
\end{equation}
where $K$ is a function from the unit ball of the space ${\rm 
Ran}(I-F^SP^{-1}(F^S)^{[*]})^{\frac{1}{2}}$. The latter space is in fact the reproducing 
kernel 
Hilbert space with reproducing kernel 
\begin{equation}
\widetilde{K}_S(z,\zeta)= K_S(z,\zeta)-F^S(z)P^{-1}F^S(\zeta)^*
\label{6.16}
\end{equation}
and the second term on the right side of \eqref{6.15} is nothing else but a function 
$h\in\cH(\widetilde{K}_S)$ such that 
\begin{equation}
\|h\|_{\cH(\widetilde{K}_S)}\le \sqrt{1-\|P^{-\frac{1}{2}}{\bf y}^*\|^2}.
\label{6.17}
\end{equation}
\begin{theorem}
\label{T:6.5}
Assume that condition \eqref{6.14} holds and that $P$ is strictly positive definite.
Let $\widetilde{K}_S$ be the kernel defined in \eqref{6.16}. Then
all solutions $f$ to the problem ${\bf AIP}_{\cH(S)}$ are described by the formula
\begin{equation}
f(z)=F^S(z)P^{-1}{\bf y}^*+h(z)
\label{6.18}
\end{equation}  
where $h$ is a free parameter from $\cH(\widetilde{K}_S)$  subject to norm constraint 
\eqref{6.17}.
The problem ${\bf AIP}_{\cH(S)}$ has a unique solution if and only if
$\|P^{-\frac{1}{2}}{\bf y}^*\|=1$ or $\widetilde{K}_S(z,\zeta)\equiv 0$.
\end{theorem}
\begin{remark} 
{\rm The function $h$ on the right hand side of \eqref{6.18} represents
in fact the general solution of the homogeneous interpolation problem
(with interpolation condition $(F^S)^{[*]}f=0$). If $h$ runs through
the whole space $\cH(\widetilde{K}_S)$, then formula \eqref{6.18}
produces all functions $f\in\cH(S)$ such that $(F^S)^{[*]}f={\bf y}^*$.
This unconstrained interpolation problem has a solution if and only if 
${\bf y}^*\in{\rm Ran} P^{\frac{1}{2}}$ and has a unique solution
if and only if $\widetilde{K}_S(z,\zeta)\equiv 0$.} 
\label{R:6.6}
\end{remark}
Parametrization of the form \eqref{6.18} is typical for interpolation problems 
in reproducing kernel Hilbert spaces. The most interesting part in this topic
is to get a more detailed characterization of all solutions of the homogeneous problem.
For the case $S\equiv 0$, such a characterization is given by Beurling-Lax theorem.
It is quite remarkable that an analog of the Beurling-Lax theorem holds in general
de Branges-Rovnyak space. For getting these analogs the assumption (3) in Definition 
\ref{D:6.1}
(which has not been used so far) is crucial.

\subsection{Analytic descriptions of the solution set} 
Following the strategy from Section 3,
the starting point is the analog of Theorem \ref{5.3} for the current Hilbert space setting. 

\begin{theorem}
A function $f: \, \D\to \cY$ is a solution of the problem ${\bf AIP}_{\cH(S)}$ with data 
set 
\eqref{6.7} if and only if the kernel
\begin{equation}
{\bf K}(z,\zeta)=\begin{bmatrix} 1 & {\bf y} & f(\zeta)^*\\
{\bf y}^* & P & F^S(\zeta)^{*} \\ f(z) & F^S(z) & K_S(z,\zeta)
\end{bmatrix}\qquad (z,\zeta\in\D),
\label{6.19}
\end{equation}
is positive on $\D\times\D$. Here $P$, $F^S$ and $K_S$ are given by
\eqref{6.6}, \eqref{5.17} and \eqref{deBRkernel}, respectively.
\label{T:6.7}
\end{theorem}
\begin{proof}
By Lemma \ref{L:6.2} specialized to $A$ and $B$ as in \eqref{6.13}
and $X=M_f$, One can now conclude that $f$ is a solution to the problem ${\bf 
AIP}_{\cH(S)}$
(that is, it meets conditions \eqref{6.8}) if and only if the following
operator is positive semidefinite:
$$
{\bf P}:=\begin{bmatrix}1 & {\bf y} & M_f^{[*]}
\\ {\bf y}^* & (F^S)^{[*]}F^S &  (F^S)^{[*]} \\ M_f &  F^S
& I_{\cH(S)} \end{bmatrix}= \begin{bmatrix}1 & {\bf y} & M_f^{[*]} \\
{\bf y}^* & P &  (F^S)^{[*]} \\ M_f &  F^S & I_{\cH(S)}
\end{bmatrix}\ge 0.
$$
As in the proof of Theorem \ref{T:5.3}, it is useful to observe that for every 
$g\in\C\oplus{\mathcal X}\oplus \cH(S)$ of the form
\begin{equation}
g(z)=\sum_{j=1}^r\left[\begin{array}{c}c_j \\ x_j\\
K_S(\cdot,z_j)y_j\end{array}\right]\qquad
(c_j\in\C, \; y_j\in\cY, \; x_j\in{\mathcal X}, \; z_j\in\D)
\label{6.20}
\end{equation}
the identity
\begin{equation}
\left\langle {\bf P}g, \; g\right\rangle_{\C\oplus\cX\oplus  
\cH(S)}=\sum_{j,\ell=1}^r \left\langle{\bf K}(z_j,z_\ell)
\sbm{c_\ell \\ x_\ell \\ y_\ell}, \; \sbm{c_j \\ x_j \\ y_j}
\right\rangle_{\C\oplus{\mathcal X}\oplus \cY}
\label{6.21}
\end{equation}
holds. Since the set of
vectors of the form \eqref{6.20} is dense in $\C\oplus{\mathcal X}\oplus
\cH(S)$, the identity \eqref{6.21} now implies that the operator
${\bf P}$ is positive semidefinite if and only if the quadratic form on the
right hand side of \eqref{6.21} is nonnegative, i.e., if and only if
the kernel \eqref{6.19} is positive on $\D\times\D$.
\end{proof}
The next observation is that for any ${\bf AIP}_{\cH(S)}$-admissible data set \eqref{6.7}, 
the Schur-class 
function  $S$ is a solution of the Schur-class problem {\bf AIP} with the data set 
$\{T,E,N,P=(F^S)^{[*]}F^S\}$. If $P$ is strictly
positive definite, then by Theorem \ref{T:5.6}, $S$ is necessarily of the form 
\eqref{5.27} for a 
$J$-inner function $\Theta$ explicitly constructed from the data set and a Schur-class 
function 
$\cE\in\cS(\cU,\cY)$
which is recovered from $S$ by the formula 
\begin{equation}
\cE=(\Theta_{11}-S\Theta_{21})^{-1}(S\Theta_{22}-\Theta_{12}).
\label{6.23}
\end{equation}
Furthermore, the formula \eqref{5.28} for the kernel $\widetilde{K}_S$ can be written in 
terms of 
this $\cE$ as 
\begin{equation}
\widetilde{K}_S(z,\zeta)=u(z)K_\cE(z,\zeta)u(\zeta)^*,
\label{6.24}  
\end{equation}
where 
\begin{equation}
K_\cE(z,\zeta)=\frac{I_\cY-\cE(z)\cE(\zeta)^*}{1-z\overline{\zeta}},\quad
u(z)=\Theta_{11}(z)-S(z)\Theta_{21}(z).
\label{6.25}  
\end{equation}

\begin{theorem}  \label{T:6.8}
Assume that the data set of the problem ${\bf AIP}_{\cH(S)}$
is such that the operator $P=(F^S)^{[*]}F^S$ is strictly positive definite.
Let $\Theta=\sbm{\Theta_{11} & \Theta_{12} \\ \Theta_{21} & \Theta_{22}}$ be a $J$-inner 
function 
satisfying \eqref{5.21a} and let $\cE\in\cS(\cU,\cY)$, $u$ and $F^S$ be given as in 
\eqref{6.23}
\eqref{6.25} and \eqref{5.17}. Then:
\begin{enumerate}
\item All solutions $f$ of the problem ${\bf AIP}_{\cH(S)}$ are parametrized by the 
formula
\begin{equation}
f=F^SP^{-1}{\bf y}^*+uh,
\label{6.26}
\end{equation}
where $h$ is a function from the de Branges-Rovnyak space $\cH(\cE)$ such that
\begin{equation}
\|h\|_{\cH(\cE)}\le \sqrt{1-\|P^{-\frac{1}{2}}{\bf y}^*\|^2}.
\label{6.27}
\end{equation}
\item Representation \eqref{6.26} is orthogonal in the metric of
$\cH(S)$.
\item The multiplication operator $M_u: \, \cH(\cE)\to \cH(S)$ is isometric.
\item  For $f$ defined by \eqref{6.26},
\begin{equation}\label{6.26a}
\|f\|^2_{\cH(S)}=\|P^{-\frac{1}{2}}{\bf y}^*\|^2_{\cX}+\|h\|^2_{\cH(\cE)}.
\end{equation}
\end{enumerate}

\end{theorem}
\begin{proof} In the case $P$ is strictly positive definite, one can  take its Schur 
complement in 
${\bf K}$ to get, on account of \eqref{6.24}, the equivalent inequality
\begin{equation}
\begin{bmatrix} 1-\|P^{-\frac{1}{2}}{\bf y}^*\|^2 &
f(\zeta)^*-{\bf y}P^{-1}F^S(\zeta)^*\\
f(z)-F^S(z)P^{-1}{\bf y}^* & u(z)K_\cE(z,\zeta)u(\zeta)^*
\end{bmatrix}\succeq 0.
\label{6.28}
\end{equation}
The latter positivity is equivalent to the function $f-F^SP^{-1}{\bf y}^*$
be of the form $uh$ for some $h\in\cH(\cE)$ subject to norm constraint \eqref{6.27}.
Statement (2) follows by Remark \ref{R:6.3} and the isometric property of the operator 
$M_u: \, 
\cH(\cE)\to \cH(S)$ is a consequence of factorization \eqref{6.24}. The last statement 
follows 
from parts (2) and (3) and the fact that $\|F^SP^{-1}{\bf 
y}^*\|^2_{\cH(S)}=\|P^{-\frac{1}{2}}{\bf y}^*\|^2_{\cX}$.
\end{proof}

\subsection{Description based on the Redheffer transform} Theorem \ref{T:6.7} holds true 
even if $P$ is not strictly positive definite. However, in this case one should use the 
Redheffer representation \eqref{5.38} for $S$ rather than \eqref{5.27}. Due to condition 
\eqref{6.14}, there exists a (unique) $\widetilde{\bf y}^*\in\cX\ominus{\rm Ker}P$ such 
that 
${\bf y}^*=P^{\frac{1}{2}}{\bf y}^*$. As in the nondegenerate case,
$S$ is a solution of the Schur-class problem {\bf AIP} with the data set
$\{T,E,N,P=(F^S)^{[*]}F^S\}$ and therefore, it is of the form \eqref{5.38} for some 
(perhaps, not 
uniquely determined) Schur-class function $\cE$. Nevertheless, identities \eqref{5.42} and 
\eqref{5.44} hold for functions $G$ and $\Gamma$ defined via formulas \eqref{5.43}, and 
making use of 
these identities the kernel \eqref{6.19} can be written as 
$$
{\bf K}(z,\zeta)=
\begin{bmatrix} 1 & \widetilde{\bf y}P^{\frac{1}{2}} & f(\zeta)^*\\
P^{\frac{1}{2}}\widetilde{\bf y}^* & P & P^{\frac{1}{2}}\Gamma(\zeta)^{*} \\ f(z) &
\Gamma(z)P^{\frac{1}{2}} & G(z)K_\cE(z,\zeta)G(\zeta)^*+\Gamma(z)\Gamma(\zeta)^*
\end{bmatrix}.
$$
The positivity of the latter kernel is equivalent to positivity of the Schur complement of 
$P$
with respect to ${\bf K}(z,\zeta)$, that is, to the condition
\begin{equation}
\begin{bmatrix}1-\|\widetilde{\bf y}\|^2 & f(\zeta)^*-\widetilde{\bf
y}\Gamma(\zeta)^{*} \\
f(z)-\Gamma(z)\widetilde{\bf y}^* &  G(z)K_\cE(z,\zeta)G(\zeta)^*
\end{bmatrix}\succeq 0\qquad (z,\zeta\in\D).
\label{6.31}
\end{equation}
It follows from the identity \eqref{5.42} that the multiplication operators
$M_G: \, h\mapsto Gh$ and $M_\Gamma: \, x\to \Gamma x$ are contractions from
$\cH(\cE)$ to $\cH(S)$ and from $\cX_0=\overline{\rm Ran}P^{\frac{1}{2}}$, respectively, 
and that
the operator
$$\begin{bmatrix} M_G & M_\Gamma\end{bmatrix}: \; \begin{bmatrix}
\cH(K_\cE) \\ \cX_0\end{bmatrix}\to \cH(K_S)
$$
is coisometric. Furthermore, since in the current case $P=(F^{S})^{[*]}F^S$,
it follows from \eqref{5.42} and \eqref{5.44} that
$M_\Gamma$ is an isometry and $M_G$ a partial isometry.
This leads to the following analog of Theorem \ref{T:6.8}.

\begin{theorem}
\label{T:6.10} 
All solutions $f$ of the problem ${\bf AIP}_{\cH(S)}$ is given by the formula
\begin{equation}\label{6.32}
f(z)=\Gamma(z)\widetilde{\bf y}^*+G(z)h(z)
\end{equation}
with parameter $h$ in $\cH(\cE)$ subject to $\|h\|_{\cH(\cE)}\leq \sqrt{1-\|\widetilde{\bf
y}\|^2}$. Furthermore, for $f$ defined by \eqref{6.32}
\begin{equation}
\|f\|^2_{\cH(S)}=\|\Gamma\widetilde{\bf y}^*\|_{\cH(S)}^2+\|Gh\|_{\cH(S)}^2
=\|\widetilde{\bf y}\|_{\cX}^2+\|P_{\cH(\cE)\ominus{\rm Ker}M_G}h\|^2_{\cH(\cE)}
\label{6.33}
\end{equation}
and hence $f_\textup{min}(z)=\Gamma(z)\widetilde{\bf y}^*$ is the unique minimal-norm 
solution.
\end{theorem}   
The latter theorem is not a complete analog of Theorem \ref{T:6.8} since (1) the 
Schur-class function $\cE$ is not determined uniquely and (2) the multiplication operator 
$M_G$ 
is not isometric. To get a closer analog of Theorem \ref{T:6.8}, it 
makes sense to assume that the operator $T$ meets the condition
\begin{equation}
\left(\bigcap_{k\ge 1}{\rm Ran} (T^*)^k\right)\bigcap{\rm Ker}T^*=\{0\}
\label{6.29}
\end{equation}
which is indeed satisfied for the following important particular cases:
\begin{enumerate}
    \item $T^{*}$ is injective (so $\operatorname{Ker} T^{*} =
    \{0\}$),
    \item $T^{*}$ is nilpotent (so $\bigcap_{k\ge 1} \operatorname{Ran} (T^{*})^{k} =
    \{0\}$), and
    \item $\dim \cX < \infty$, or, more generally e.g., $T = \lambda
    I + K$ with $0 \ne \lambda \in {\mathbb C}$ and $K$ compact (so
    $\cX = \operatorname{Ran} (T^{*})^{p} \dot + \operatorname{Ker}
    (T^{*})^{p}$ once $p$ is sufficiently large).
\end{enumerate}

\begin{theorem}\label{T:6.11}
Let $\Sigma$ be the Schur-class
function defined in \eqref{3.7} from the ${\bf AIP}_{\cH(S)}$-admissible data set
\eqref{6.7} and $P=(F^{S})^{[*]}F^S$. If $T$ satisfies condition \eqref{6.29}, then
\begin{enumerate}
\item The Redheffer transform ${\mathcal R}_{\Sigma}: \, 
\cS(\widetilde{\Delta},\widetilde{\Delta}_*)\to\cS(\cU,\cY)$ defined by formula 
\eqref{5.38}
is injective.
\item For any $\cE\in\cS(\widetilde{\Delta},\widetilde{\Delta}_*)$, the multiplication 
operator $M_G: \, f\to Gf$ by the function  $G=\Sigma_{12}(I-\cE\Sigma_{22})^{-1}$
is an isometry from $\cH(\cE)\to\cH(S)$.
\end{enumerate}
\end{theorem}
Thus, if the operator $T$ in the ${\bf AIP}_{\cH(S)}$-admissible data set
\eqref{6.7} satisfies condition \eqref{6.29} then the Schur-class function $\cE$ in 
Theorem 
\ref{T:6.10} is determined uniquely from  the data set and $\Sigma$, and the formula
\eqref{6.33} takes a simpler form 
$$
\|f\|^2_{\cH(S)}=\|\widetilde{\bf y}\|_{\cX}^2+\|h\|^2_{\cH(\cE)}.
$$

\section{Boundary behavior and boundary interpolation in $\cH(S)$}
In this section, relations between boundary regularity of a Schur-class function $S$ 
and boundary regularity of functions in the associated   
de Branges-Rovnyak space $\cH(S)$ will be discussed. The next theorem presents a result of 
this type. 
\begin{theorem}
Let $I$ be an open arc of $\mathbb T$, let $S$ be a scalar-valued Schur-class function and 
let 
$A_S=R_0\vert_{\cH(S)}$ be the model operator for $S$ (see \eqref{realization}). The 
following 
are equivalent:
\begin{enumerate}
\item $S$ admits an analytic continuation across $I$ and $|S(\zeta)|=1$ for all $\zeta\in 
I$.
\item $I$ is contained in the resolvent set of $A_S^*$.
\item Any function $f\in\cH(S)$ admits an analytic continuation across $I$.
\end{enumerate}
\label{T:7.1}
\end{theorem}

For the proof the reader is referred to \cite{frma} and to earlier sources
\cite[Chapter VIII]{helson} (for the case where $S$ is inner) and 
\cite[Chapter 5]{sarasonbook} (for the case where $S$ is subject to 
the condition $\int_{\mathbb T} \log 
(1-|S(\zeta)|)d\zeta=-\infty$). 

\smallskip

The single-point (local) version of Theorem \ref{T:7.1} is the following: 
\begin{theorem}
Let $S$ be a scalar-valued Schur-class function and let $t_0\in\mathbb T$.
The following are equivalent:
\begin{enumerate}
\item $S$ admits an analytic continuation into a neighborhood $\cU$ of $t_0$
and is unimodular on $\cU\cap \mathbb T$.
\item The operator $t_0 I-A_S^*$ is invertible on $\cH(S)$.
\item Any function $f\in\cH(S)$ admits an analytic continuation into $\cU$.
\end{enumerate}
\label{T:7.2}  
\end{theorem}
Our next aim is to find conditions (in terms of $S$) guaranteeing a weaker but more 
natural 
property: the existence of the nontangential boundary limit 
\begin{equation}
f(t_0):=\angle \lim_{z\to t_0}f(z) 
\label{7.1}
\end{equation}
(i.e., $t_0$ is approached from within an arbitrary but  fixed Stolz angle with the 
vertex at $t_0$) for any function $f\in\cH(S)$. Upon comparing statements
(1) and (3) in Theorems \ref{T:7.1} and \ref{T:7.2}, one may think that 
the desired condition might be that $S$ admits a unimodular boundary limit 
$S(t_0)$. This condition is indeed necessary but not sufficient as 
the next theorem shows.
    \begin{theorem}\label{T:7.3}
    Let $S\in\cS$ and $t_0\in\mathbb T$. The following are
    equivalent:
    \begin{enumerate}
\item The boundary limit \eqref{7.1} exists for every $f\in\cH(S)$.
\item $S$ meets the Carath\'eodory-Julia condition
\begin{equation}
\liminf_{z\to t_0}\frac{1-|S(z)|^2}{1-|z|^2}<\infty
\label{7.2}
\end{equation}
where $z$ tends to $t_0$ unrestrictedly in $\D$.
\item The boundary limits 
\begin{equation}
S_0=\angle \lim_{z\to t_0}S(z)\quad\mbox{and}\quad 
S_1=\angle \lim_{z\to t_0}S^\prime (z)
\label{7.3}
\end{equation}
exist and are subject to conditions 
\begin{equation}
|S_0|=1\quad\mbox{and}\quad t_0S_1\overline{S}_0\in{\mathbb R}.
\label{7.4}   
\end{equation}
\item The boundary limit $S_0$ in \eqref{7.3} exists
and the function
$K_{t_0}(z):={\displaystyle\frac{1-S(z)\overline{S}_0}{1-z\overline{t}_0}}$
    belong to $\cH(S)$.
    \end{enumerate}
    Moreover, if the conditions (1)--(4) are satisfied then 
$$
\liminf_{z\to t_0}\frac{1-|S(z)|^2}{1-|z|^2}=\angle\lim_{z\to 
t_0}\frac{1-|S(z)|^2}{1-|z|^2}=
t_0S_1\overline{S}_0=\|K_{t_0}\|^2_{\cH(S)}\ge 0
$$
and the function $K_{t_0}$  is the boundary reproducing kernel
    in $\cH(S)$ in the sense that
    \begin{equation}
    \langle f, \, K_{t_0}\rangle_{\cH(S)}=f(t_0):=\angle\lim_{z\to
    t_0}f(z).  
    \label{7.5}
    \end{equation}
    \end{theorem}
Equivalence $(2) \Leftrightarrow (3)$ is the classical Carath\'eodory-Julia theorem
on angular derivatives (see e.g., \cite[Chapter 4]{shapiro}). Note that condition 
\eqref{7.2} is equivalent to the requirement that the function 
$z\mapsto K_S(z,\zeta)$ stay bounded in the norm of $\cH(S)$ as $\zeta$ tends to $t_0$ 
unrestrictedly in $\D$.

\smallskip

The equivalences $(1) \Leftrightarrow(2) \Leftrightarrow(4)$ are proved in \cite[Chapter 
6]{sarasonbook}. Statement (4) is presented in \cite{sarasonbook} in a seemingly weaker 
form 
\begin{enumerate}
\item[($4^\prime$)] {\em There is a number $\lambda\in\C$ such that the function 
$\frac{1-S(z)\overline\lambda}{1-z\overline{t}_0}$ belongs to $\cH(S)$}. 
\end{enumerate}
It then can be shown that there is a unique such $\lambda$ which  turns out to be equal
to the boundary limit $S_0$ as in \eqref{7.3}.

\smallskip

The list of equivalent conditions in Theorem \ref{T:7.3} can be extended by several other 
ones.
The extended list of such equivalences will be presented in the context of a  a more 
general question:
given a Schur-class function $S$, given a point $t_0\in\mathbb T$ and  and given an 
integer $n\ge 
0$, find conditions necessary  and sufficient for the existence of boundary limits 
 \begin{equation} 
f_j(t_0):=\angle\lim_{z\to t_0}\frac{f^{(j)}(z)}{j!} \quad\mbox{for}\quad j=0,\ldots,n
   \label{7.6} 
    \end{equation}
and for any function $f\in\cH(S)$.

\begin{theorem}\label{T:7.4}
Let $s\in\cS$, $t_0\in\mathbb T$ and $n\in{\mathbb N}$. The following are
    equivalent:
    \begin{enumerate}
\item[(1)] The boundary limits \eqref{7.6} exist for every $f\in\cH(S)$.
\item[(1$^\prime$)] The boundary limit ${\displaystyle \angle\lim_{z\to t_0}f^{(jn}(z)}$ 
exists
for every $f\in\cH(S)$.
\item[(2)] $S$ meets the generalized Carath\'eodory-Julia condition
 \begin{equation}
    \liminf_{z\to t_0}\frac{\partial^{2n}}{\partial
    z^{n}\partial\bar{z}^{n}} \, \frac{1-|S(z)|^2}{1-|z|^2}<\infty.
    \label{7.7}
    \end{equation}
Equivalently, the function $\frac{\partial^n}{\partial\bar{\zeta}^n}K_S(\cdot,\zeta)$ 
stays
bounded in the norm of $\cH(S)$ as $\zeta$ tends radially to $t_0$.
 
\item[(3)] The boundary limits $S_j:=S_j(t_0)$ exist for $j=0,\ldots,2n+1$
and are such that $|S_0|=1$ and the matrix
    \begin{equation}
    {\mathbb P}^S_{n}(t_0):=\left[\begin{array}{ccc} S_1 & \cdots &
    S_{n+1} \\ \vdots & &\vdots \\
    S_{n+1} & \cdots & S_{2n+1}\end{array}\right]{\bf \Psi}_n(t_0)   
    \left[\begin{array}{ccc}\overline{S}_0 & \ldots &
    \overline{S}_n\\ & \ddots & \vdots \\ 0 &&\overline{S}_0
    \end{array}\right]
    \label{7.8}
    \end{equation}
    is Hermitian, where the first factor is a Hankel matrix, the
    third factor is an upper triangular Toeplitz matrix and where
    ${\bf \Psi}_n(t_0)$ is the upper triangular matrix given by
    \begin{equation}
    {\bf\Psi}_n(t_0)=\left[\Psi_{j\ell}\right]_{j,\ell=0}^n,\quad
    \Psi_{j\ell}=(-1)^{\ell}\left(\begin{array}{c} \ell \\ j
    \end{array}\right)t_0^{\ell+j+1},\quad 0\le j\leq\ell\le n.
    \label{7.9}
    \end{equation}
\item[(4)] The boundary limits $S_j:=S_j(t_0)$ exist for $j=0,\ldots,n$ and the
    functions
    \begin{equation}
    K_{t_0,j}(z):=\frac{z^j}{(1-z\overline{t}_0)^{j+1}}-S(z)\cdot
\sum_{\ell=0}^j\frac{z^{j-\ell}\overline{S}_\ell}{(1-z\overline{t}_0)^{j+1-\ell}}
    \label{7.10}
    \end{equation}  
    belong to $\cH(K_S)$ for $j=0,\ldots,n$. Equivalently, the limits $S_j:=S_j(t_0)$ 
exist for 
$j=0,\ldots,n$ and the single function $K_{t_0,n}$ belongs to $\cH(K_S)$.

\item[(4$^\prime$)] There exist complex numbers $\lambda_0,\ldots,\lambda_n$ such that 
the function
$$
\frac{z^j}{(1-z\overline{t}_0)^{j+1}}-S(z)\cdot
\sum_{\ell=0}^n\frac{z^{n-\ell}\overline{\lambda}_\ell}{(1-z\overline{t}_0)^{n+1-\ell}}
$$
 belongs to $\cH(K_s)$.

\item[(5)]  It holds that
    \begin{equation} \label{7.11}
     {\displaystyle\sum_k\frac{1-|a_k|^2}{|t_0-a_k|^{2n+2}}+
    \int_0^{2\pi}\frac{d\mu(\theta)}{|t_0-e^{i\theta}|^{2n+2}}<\infty}
    \end{equation}
    where the numbers $a_k$ come from the Blaschke product of the inner-outer 
factorization of
$S$:
    $$
    S(z)=\prod_k\frac{\bar{a}_k}{a_k}\cdot\frac{z-a_k}{1-z\bar{a}_k}\cdot
    \exp\left\{-\int_0^{2\pi}\frac{e^{i\theta}+z}{e^{i\theta}-z}d
    \mu(\theta)\right\}.
   $$
\item[(6)] $(A_S^*)^nK_S(\cdot,0)$ belongs to the range of $(I-\overline{t}_0A_S^*)^{n+1}$
where $A_S=R_0\vert_{\cH(S)}$ is the model operator for $S$.

\item[(7)] There exists a finite Blaschke product $b$ such that 
\begin{equation}
S(z)=b(z)+o(|z-t_0|^{2n+1}).
\label{7.12}
\end{equation} 
as $z$ tends to $t_o$ nontangentially. 
\item[(8)] Asymptotic equality \eqref{7.12} holds for a 
rational function $b$ which is unimodular on $\mathbb  T$
(i.e., $b$ is the ratio of two finite Blaschke products). 
    \end{enumerate}
    Moreover, if conditions (1)--(8) are satisfied, and hence all, then:
    \begin{enumerate}
    \item[(a)] The matrix \eqref{7.8} is positive semidefinite and
    equals
    \begin{equation}
    {\mathbb P}^S_{n}(t_0)=\left[\langle K_{t_0,i}, \,
    K_{t_0,j}\rangle_{\cH(S)}\right]_{i,j=0}^n.
    \label{7.13}
    \end{equation}
    \item[(b)] The functions \eqref{7.10} are boundary reproducing kernels
    in $\cH(S)$ in the sense that
    \begin{equation}  
    \langle f, \, K_{t_0,j}\rangle_{\cH(S)}=f_j(t_0):=\lim_{z\to
    t_0}\frac{f^{(j)}(z)}{j!}\quad\mbox{for}\quad j=0,\ldots,n.
    \label{7.14}
    \end{equation}
    \end{enumerate}
    \end{theorem}   
Statements (2)$ \Leftrightarrow$(4)$\Leftrightarrow$(4$^\prime$),
implication (1)$\Rightarrow$(3) and statements (a) and (b) were proved in \cite{bk3}; 
implication (3)$\Rightarrow$(1) and equivalences 
(1)$\Leftrightarrow$(7)$\Leftrightarrow$(8) 
appear in \cite{bk4}. Equivalences (1)$\Leftrightarrow$(5)
$\Leftrightarrow$(6) were established in \cite{acl} for $S$ inner and extended in 
\cite{frma} to general Schur-class functions. The implication (1)$ 
\Rightarrow$(1$^\prime$)
is trivial while the converse implication will be clarified in Lemma \ref{L:7.5} below.
Finally, it is not hard to see that for $n=0$,
the statements (1)--(4$^\prime$) in Theorem \ref{T:7.4} amount to the respective 
statements in 
Theorem \ref{T:7.3}.
\subsection{Boundary interpolation}
Theorem \ref{T:7.4} suggests a boundary interpolation problem: 

\medskip

${\bf BP}_{\cH(S)}$: {\em Given a Schur-class function 
$S$ satisfying the Carath\'eodory-Julia condition \eqref{7.7} at $t_0\in\mathbb T$ 
(or one of the equivalent conditions from Theorem \ref{T:7.3}) and given complex numbers 
$f_0,\ldots,f_n$, find all $f\in\cH(S)$
such that $\|f\|_{\cH(S)}\le 1$ and
\begin{equation}
f_j(t_0):=\angle\lim_{z\to t_0}\frac{f^{(j)}(z)}{j!}=f_j\quad\mbox{for}\quad
 j=0,\ldots,n.
\label{7.15}
\end{equation}}
 According to Theorem \ref{T:7.4}, condition \eqref{7.7} guarantees that
    all the boundary limits in \eqref{7.15} exist as well as the boundary
    limits  $S_j:=S_j(t_0)$ exist for $j=0,\ldots,2n+1$. Introduce the matrices 
\begin{equation}
T=\left[\begin{array}{cccc} \bar{t}_0 & 1 & \ldots & 0 \\
    0 &  \bar{t}_0 & & \vdots \\ \vdots & \ddots & \ddots& 1\\
    0 & \ldots & 0 & \bar{t}_0\end{array}\right]\quad\mbox{and}\quad 
\begin{bmatrix}E \\ N \\ {\bf y}\end{bmatrix}= \left[\begin{array}{cccc} 1& 0& \ldots&0 \\ 
\overline{S}_{0}&    \overline{S}_{1}& \ldots & \overline{S}_{n}\\
    \overline{f}_{0}&    \overline{f}_{1}& \ldots & \overline{f}_{n}
    \end{array}\right].
\label{7.16}   
\end{equation}
Simple matrix computations show that 
$$
\left[\begin{array}{c} E \\ N
\end{array}\right]\left(I-zT\right)^{-1}=
\left[\begin{array}{ccc}{\displaystyle\frac{1}{1-z\bar{t}_0}} &
    \ldots & {\displaystyle\frac{z^{n}}{(1-z\bar{t}_0)^{n+1}}}\\
    {\displaystyle\frac{\overline{S}_{0}}{1-z\bar{t}_0}} & \ldots &
    {\displaystyle\sum_{\ell=0}^{n}\frac{\overline{S}_{\ell}z^{n-\ell}}
    {(1-z\bar{t}_0)^{n+1-\ell}}}\end{array}\right].
$$
    Multiplying the latter equality by $\begin{bmatrix} 1 &
    -S(z)\end{bmatrix}$ on the left  and taking into account formulas \eqref{7.10} for 
    $K_{t_0,j}$  one gets
\begin{equation}
F^S(z):=(E-S(z)N)(I-zT)^{-1}=\begin{bmatrix} K_{t_0,0}(z) &  K_{t_0,1}(z)&\ldots & 
 K_{t_0,n}(z)\end{bmatrix}.
\label{7.17}
    \end{equation}
The next task is to show that the problem {\bf AIP}$_{\cH(S)}$
with $\cX=\C^{n+1}$ and $\{S, T, E, N, {\bf y}\}$ taken in the form \eqref{7.16}
is equivalent to the problem {\bf BP}$_{\cH(S)}$. First it will be shown 
that the data is ${\bf AIP}_{\cH(S)}$-admissible.

\smallskip

The first requirement in Definition \ref{D:6.1} is self-evident since
$(I-zT)^{-1}$ is a rational function with no poles inside $\D$.
However, it is worth noting that the pair $(E,T)$ is not output-stable
and so {\bf BP}$_{\cH(S)}$ cannot be embedded into the scheme of the
problem ${\bf IP}_{\cH(S)}$. 
    
\smallskip

For a generic vector  $x=\operatorname{Col}_{0\le j\le n}x_{j}$ in  $\cX=\C^{n+1}$,
It follows from \eqref{7.17},
\begin{equation}
F^S(z)x=\sum_{j=0}^n K_{t_0,j}(z)x_j
\label{7.18}
\end{equation}
and the latter function belongs to $\cH(S)$ by statement (4) in Theorem \ref{T:7.4}.
It then follows from \eqref{7.13} that $\left\langle  {\mathbb P}^S_{n}(t_0)x, \, 
x\right\rangle_{\cX}=\|F^Sx\|^2_{\cH(S)}$ and therefore 
${\mathbb P}^S_{n}(t_0)=(F^S)^{[*]}F^S$. It was shown in \cite{bk3} that 
the structured matrix ${\mathbb P}^S_{n}(t_0)$ of the form \eqref{7.8} satisfies the Stein 
identity 
$$
{\mathbb P}^S_{n}(t_0)-T^*{\mathbb P}^S_{n}(t_0)T=E^*E-N^*N.
$$
Therefore, the data set $\{S, T, E, N, {\bf y}\}$ is ${\bf AIP}_{\cH(S)}$-admissible.
Observe next that by \eqref{7.14} and \eqref{7.18},  for  every 
$f\in\cH(S)$ the following condition holds:
    $$
    \langle M_{F^S}^{[*]}f, \; x\rangle_{\cX}=\langle f, \;
    M_{F^S}x\rangle_{\cH(S)}=\sum_{j=0}^{n}f_j(t_0)\overline{x}_{j}.
$$
    On the other hand, for ${\bf y}$ defined in \eqref{7.16}, 
    $\langle {\bf y}^*, \, x\rangle_{\cX}=\sum_{j=0}^{n}f_{j}\overline{x}_{j}$. 
Therefore interpolation conditions \eqref{7.15} are equivalent to the equality
    $$
    \langle M_{F^s}^{[*]}f, \; x\rangle_{\cX}=\langle {\bf y}^*, \,
    x\rangle_{\cX}
    $$
holding for every $x\in\cX$, i.e., to equality $M_{F^s}^{[*]}f={\bf y}^*$. 
Thus the problem {\bf AIP}$_{\cH(S)}$
with the data set $\{S, T, E, N, {\bf y}\}$ taken in the form \eqref{7.16},   
is equivalent to the {\bf BP}$_{\cH(S)}$. In particular, the problem {\bf BP}$_{\cH(S)}$
has a solution if and only if ${\mathbb P}^S_{n}(t_0)\ge {\bf y}^*{\bf y}$ and all 
solutions to 
the problem are parametrized as in Theorem \ref{T:6.10}.

\subsection{The vector-valued case} 
In the vector-valued setting, two additional issues need to be 
addressed. Firstly, if $\dim\cY=\infty$, one should specify the topology with respect to 
which the boundary 
limits should converge. This issue is easily resolved due to the following result; see 
\cite[Lemma 2.1]{bk5} 
for 
the proof. 

\begin{lemma}
Let $n\in\mathbb N$, $t_0\in\mathbb T$, let $f$ be an $\cL(\cU,\cY)$-valued function 
analytic on ${\mathcal U}_{t_0,\varepsilon}=\{z\in\D: \, 0<|z-t_0|<\varepsilon\}$
and assume that for any $\alpha\in(0,\frac{\pi}{2})$ there exists 
$\gamma_\alpha<\infty$ such that
$$
\|f^{(n)}(z)\|\le \gamma_{\alpha}\qquad\mbox{for all}\quad  
z\in \{z\in{\mathcal U}_{t_0,\varepsilon}: \,|{\rm arg}(z-t_0)|<\alpha\}.
$$
Then  the uniform limits $\angle{\displaystyle
\lim_{z\to t_0}f^{(j)}(z)}$ exist for $j=0,\ldots,n-1$.

In particular, the statement holds if the weak limit
$\angle {\displaystyle\lim_{z\to t_0}f^{(n)}(z)}$ exists.
\label{L:7.5}
\end{lemma}
Thus, once the existence of the nontangential boundary limit (even in the weak sense) for 
the $n$-th 
derivative
of any function $f\in\cH(S)$ is settled, the existence of strong nontangential boundary 
limits for derivatives 
of lower order will be settled automatically.

\smallskip

Secondly, in the vector-valued case, one may want to guarantee the existence of the full  
or just of a 
tangential boundary limit. The formulation which incorporates both options appears to be 
the following.:
{\em Given a Schur-class function $S\in\cS(\cU,\cY)$ and given an $\cL(\cY,\cG)$-valued 
polynomial  
\begin{equation}
A(z)=\sum_{j=0}^n A_j(z-t_0)^j,
\label{7.18a}
\end{equation}
find conditions which are necessary and sufficient for the existence of strong boundary 
limits 
 \begin{equation}
\angle\lim_{z\to t_0}(Af)^{(j)}(z) \quad\mbox{for}\quad j=0,\ldots,n
   \label{7.19}
    \end{equation}
and for any function $f\in\cH(S)$}. The answer is given in the following theorem.
\begin{theorem}
The strong limits \eqref{7.19} exist for any function $f\in\cH(S)$ id and only if
\begin{equation}
\liminf_{z\to t_0}\left\langle
\frac{\partial^{2n}}{\partial z^n\partial\bar{z}^n}
\left(A(z)\frac{I_{\cY}-S(z)S(z)^*}{1-|z|^2}A(z)^*\right)
g,\, g\right\rangle<\infty 
\label{7.20}
\end{equation}
for every $g\in\cG$. If this is the case, then the limits 
\begin{equation}
b_j:=\angle \lim_{z\to t_0} \frac{(AS)^{(j)}(z)}{j!}\qquad (j=0,\ldots,n)
\label{7.21}
\end{equation}
exist in the strong sense and  the limit 
\begin{equation}
P=\angle\lim_{z\to t_0}\frac{\partial^{2n}}{\partial z^n\partial\bar{z}^n}
\left(A(z)\frac{I_{\cY}-S(z)S(z)^*}{1-|z|^2}A(z)^*\right)
\label{7.22}
\end{equation}
exists in the weak sense.
\label{T:7.6}
\end{theorem}
For the proof and for more equivalent reformulations and consequences of the 
Carath\'eodory-Julia condition
\eqref{7.20} (that is, for the operator-valued version of the Carath\'eodory-Julia 
theorem),
a good reference is \cite{bk5}; see also  \cite{bd2} for the 
matrix-valued case. Note that the case $n=0$
in the matrix-valued setting was studied earlier in \cite{dymde, kov3} and \cite[Section 
8]{dym}, 

\smallskip

Observe that upon choosing $A(z)\equiv a_0\in\cL(\cY,\cG)$ one can derive from Theorem 
\ref{T:7.6} the 
necessary and 
sufficient condition for the existing  of the boundary limit 
${\displaystyle\angle\lim_{z\to t_0}a_0f^{(n)}(z)}$ for all 
$f\in\cH(S)$.  Finally, here is a  formulation of the 
vector-valued analog of the problem  ${\bf BP}_{\cH(S)}$ from the previous section:

\medskip

{\em Given a Schur-class function
$S\in\cS(\cU,\cY)$ and an $\cL(\cY,\cG)$-valued polynomial $A$  subject to the 
Carath\'eodory-Julia 
condition \eqref{7.20} at  $t_0\in\mathbb T$ and given vectors $f_0,\ldots, f_n\in\cY$, 
find all 
$f\in\cH(S)$ such that $\|f\|_{\cH(S)}\le 1$ and
\begin{equation}
f_j(t_0):=\angle\lim_{z\to t_0}\frac{f^{(j)}(z)}{j!}=f_j\quad\mbox{for}\quad
 j=0,\ldots,n.
\label{7.23}
\end{equation}}
It turns out that as in the scalar-valued case, this problem is equivalent to the problem 
${\bf AIP}_{\cH(S)}$ with $\cX=\cG^{n+1}$ and with the data set 
$$
T=\left[\begin{array}{cccc} \bar{t}_0 I_\cG& I_\cG & \ldots & 0 \\
    0 &  \bar{t}_0 I_\cG & & \vdots \\ \vdots & \ddots & \ddots& I_\cG\\
    0 & \ldots & 0 & \bar{t}_0 I_\cG\end{array}\right]\quad\mbox{and}\quad
\begin{bmatrix}E \\ N \\ {\bf y}\end{bmatrix}= \left[\begin{array}{cccc} a_0^*& a_1^*& 
\ldots& a_n^* \\
b_{0}^*& b_{1}^*& \ldots & b_{n}^*\\
f^*_{0}&  f^*_{1}& \ldots & f^*_{n}
    \end{array}\right]
$$
where the operators $b_0,\ldots,b_n\in\cL(\cU,\cG)$ are defined in \eqref{7.21}. This data 
set turns out
to be ${\bf AIP}_{\cH(S)}$ admissible;.  Details are omitted here; 
note only that the operator $P=(F^S)^{[*]}F^S$ appears be equal to that  in \eqref{7.23},

\section{Concluding remarks}

The preceding sections give an overview of some of the most recent 
applications of de Branges-Rovnyak spaces to a variety of problems in 
function theory, in particular, in interpolation theory.  As the 
following examples illustrate, there is still ongoing work pushing the 
theory in still more directions.

\subsection{Canonical de Branges-Rovnyak functional-model spaces:  
multivariable settings} \label{S:multivariable} 

Realization  of a Schur-class function as the transfer function of a 
canonical functional-model colligation having additional metric 
properties (e.g., coisometric, isometric, or unitary), has been extended to 
settings where the unit disk playing the role of the underlying 
domain is replaced by a more 
general domain $\cD$ in ${\mathbb C}^{d}$; see \cite{bbBall} for the case of the unit ball 
${\mathbb B}^{d}$ in ${\mathbb C}^{d}$,  \cite{bbPoly} for the case 
of the unit polydisk ${\mathbb D}^{d}$, \cite{bbGen} for the case of 
a general domain with matrix polynomial defining function.

\smallskip

\subsection{Extensions to Kre\u{\i}n space settings}
Much of the theory of de Branges-Rovnyak spaces actually extends to Pontryagin and 
Kre\u{\i}n-space settings, where  Hilbert spaces coming up in 
various places are allowed to be Kre\u{\i}n spaces (i.e., the space 
is a direct sum of a Hilbert space and an anti-Hilbert space), or at 
least Pontryagin spaces (where the anti-Hilbert space is finite 
dimensional).  

\smallskip

The AIP approach to interpolation has been extended to the 
Kre\u{\i}n-space setting in work of Derkach \cite{Der1, Der2}; this 
includes a Pontryagin-space formulation of the  Nikolskii-Vasyunin 
model space $\widetilde \cD(S)$ in terms of Kre\u{\i}n-Langer 
representations.

\end{document}